\UseRawInputEncoding
\documentclass[11pt,leqno]{amsart}
\usepackage{txfonts}
\usepackage{amsmath,amsthm,amsfonts}
\usepackage{latexsym}
\usepackage{amssymb}
\usepackage{mathrsfs}
\usepackage{appendix}
\usepackage[colorlinks,
            linkcolor=red,
            anchorcolor=blue,
            citecolor=green,
            ]{hyperref}
\usepackage[dvips]{epsfig}

\newcounter{RomanNumber}

\renewcommand{\thesection}{\arabic{section}}

\newtheorem{theorem}{Theorem}[]
\newtheorem{lemma}{Lemma}[section]
\newtheorem{prop}{Proposition}[section]
\newtheorem{defi}{Definition}[section]

\newtheorem{corollary}{Corollary}[section]
\theoremstyle{remark}
\newtheorem{remark}{Remark}[section]

\renewcommand{\theequation}{\thesection .\arabic{equation}}
\let\sect\section
\renewcommand\section{\setcounter{equation}{0}
\gdef\theequation{\thesection .\arabic{equation}}\sect}

\makeatletter

\newcommand{\Rmnum}[1]{\expandafter\@slowromancap\romannumeral #1@}
\makeatother

\newcommand{\cB}{{\mathcal{B}}}
\newcommand{\cD}{{\mathcal{D}}}
\newcommand{\cE}{{\mathcal{E}}}

\newcommand{\cN}{{\mathcal{N}}}

\newcommand{\cG}{{\mathcal{G}}}

\newcommand{\cS}{{\mathcal{S}}}

\newcommand{\cP}{{\mathcal{P}}}
\newcommand{\cQ}{{\mathcal{Q}}}

\newcommand{\C}{{\mathbb{C}}}
\newcommand{\IC}{{\mathbb{C}}}

\newcommand{\R}{{\mathbb{R}}}
\newcommand{\IR}{{\mathbb{R}}}
\newcommand{\TT}{{\mathbb{T}}}

\newcommand{\tor}{\TT}

\newcommand{\be}{\begin{eqnarray}}
\newcommand{\ee}{\end{eqnarray}}

\newcommand{\diam}{\mathop{\rm{diam}}}

\newcommand{\dist}{\mathop{\rm{dist}}}

\newcommand{\mes}{\mathop{\rm{mes}\, }}

\newcommand{\car}{\mathop{\rm{Car}}\nolimits}

\newcommand{\spec}{\mathop{\rm{spec}}}

\def\beeq{\begin{equation}}
\def\eneq{\end{equation}}

\def\bm{\begin{matrix}}
\def\endm{\end{matrix}}
\def\cS{{\mathcal S}}

\def\Im{{\rm Im}}

\newcommand{\mbt}{\mathbb{T}}

\begin{document}

\title[ homogeneous spectrum of Gevrey Schr\"odinger   operator ]{ Homogeneous Spectrum of  Quasi-periodic Gevrey Schr\"odinger  Operators with  Diophantine  Frequency}

\author{Yan Yang}
	\address{College of Sciences, Hohai University, 1 Xikang Road Nanjing Jiangsu 210098 P.R.China}
	\email{yanyang$\underline{~}$lgz@163.com}

\author{ Kai Tao$^*$}
	\address{College of Sciences, Hohai University, 1 Xikang Road Nanjing Jiangsu 210098 P.R.China}
	\email{ktao@hhu.edu.cn,\ tao.nju@gmail.com}

\thanks{2020 Mathematics Subject Classification. 37A30, 47A35, 37C55}

 \keywords{Homogeneous Spectrum; Gevrey Potential; Diophantine Frequency;  Quasi-periodic  Schr\"odinger Operator; Large Coupling Number.}

\thanks{$^*$ Corresponding author}

\date{}

\begin{abstract}
We consider the    quasi-periodic  Schr\"odinger operator  with the non-degenerate Gevrey potential for the  Diophantine frequency. We prove that if the coupling number of the potential  is large, then the spectrum  is
homogeneous.
\end{abstract}

\maketitle

\section{Introduction}

In this paper, we  study the following quasi-periodic  Schr\"odinger operators $H(x,\omega)$ on
$l^2(\mathbb{Z})$:
\begin{equation}\label{jacobiequ}
\left(H(x,\omega)\phi\right)(n)=\phi(n+1)+\phi(n-1)+V(x+n\omega)\phi(n)
,\ n\in\mathbb{Z},\end{equation}
where $V:\mathbb{T}\to\mathbb{R}$ is a real function called potential,  $\omega$ is an irrational number called frequency and $x\in \TT$ is called phase.

Due to the rich backgrounds in quantum physics, this operator and its special example, Almost Mathieu operator(AMO for short),
\begin{equation}\label{amo}
\left(M(x,\omega)\phi\right)(n)=\phi(n+1)+\phi(n-1)+2\lambda \cos\left(2\pi (x+n\omega)\right)\phi(n)
,\end{equation}
have been extensively studied, especially  people found that one can use ideas from the
dynamical systems (mainly linear cocycles) to study them in 2000's(such as \cite{J99,BG00,GS01}). In this field, the most famous question is  the Ten Martini Problem, which was dubbed by Barry Simon and conjectures that for
any irrational $\omega$, the spectrum of AMO (\ref{amo}) is a Cantor set. It was completely solved by Avila and Jitomirskaya in \cite{AJ09}.  For the Schr\"odinger operators (\ref{jacobiequ}), Goldstein and Schlag
\cite{GS11} also obtained this Cantor spectrum, when the potential $V$ is analytic and the Lyapunov exponent $L(E,\omega)$, which will be defined later, is positive. We denote these spectrums by $\mathcal{S}_\omega$, since they don't depend on $x$ for any irrational $\omega$.  These two results show that $\cS_\omega$ has the isolation: for any $E\in \cS_\omega$, there exists a deleted neighborhood $\cN(E)$ such that $\cN(E)\bigcap \cS_\omega=\emptyset$.

On the other hand,  Goldstein-Damanik-Schlag-Voda \cite{GDSV} proved that $\cS_\omega$ also  has a certain kind of continuity in measure: they considered the strong Diophantine frequency(SDC for short), i.e. for some $\alpha>1$,
\begin{equation}\label{sdc}
\|n\omega\| \geq \frac{c_{\omega}}{n(\log |n|+1)^\alpha}\ \ \mbox{for
all}\ n\not=0,\end{equation}
and then showed that in the supercritical region, which means $L(E,\omega)>0$, $\mathcal{S}_\omega$ is homogeneous. Here,  a closed set $\mathcal{S}\subset\mathbb{R}$ is called homogeneous if there is $\tau > 0$ such that for any $E\in \mathcal{S}$ and  $0<\sigma\le \diam(\cS)$,
\begin{equation} \label{1homogeneous}
|\mathcal{S}\cap(E-\sigma,E+\sigma)| > \tau\sigma.
\end{equation}

The homogeneity of the spectrum also plays an essential role in the inverse spectral
theory of almost periodic potentials(refer to the fundamental work of Sodin-
Yuditskii \cite{SY95,SY97}). It was shown
that the homogeneity of the spectrum implies the almost periodicity of the
associated potentials. What's more, it
is deeply related to Deift's conjecture \cite{BDGL15,DGL17}, who asked
whether the solutions of the KdV equation are almost periodic if the
initial data is almost periodic. So after the work \cite{GDSV} was submitted on arXiv in 2015, the research of the homogeneous spectrum  becomes a hot spot in our field.

In \cite{LYZZ17}, Leguil-You-Zhao-Zhou proved that for a measure-theoretically typical analytic potential, $\cS_{\omega}$ is homogeneous with the SDC frequency.
For the AMO, they further proved that if $\beta(\omega)=0$ and $\lambda\not=1$, then $\cS_{\omega}$
is also homogeneous. Note that for any irrational $\omega$, there exist its continued fraction approximants $\left\{\frac{p_s}{q_s}\right\}_{s=1}^{\infty}$,  satisfying
\begin{equation}\label{irtor0}
 \frac{1}{q_s(q_{s+1}+q_s)}<\left|\omega-\frac{ p_s}{q_s}\right|<\frac{1}{q_sq_{s+1}}.
\end{equation}Then,  $\beta(\omega)$ is defined as the exponential growth exponent of $\{\frac{p_s}{q_s}\}_{s=1}^{\infty}$ as follows:
\[
 \beta(\omega):=\limsup_s \frac{\log q_{s+1}}{q_s}\in[0,\infty].
\]
Obviously, the  strong Diophantine number is  a subset of the set $\{\omega:\beta(\omega)=0\}$.  Lately, Liu-Shi \cite{LS19} extended it to the finite Liouville frequency $0\le \beta(\omega)<\infty$, but the coupling number $\lambda$ of the potential $V=\lambda V_0$
needs to be very small. Due to Avila's  global theory of one-frequency cocycles, Liu-Shi's operator is in the subcritical regime and has purely absolutely continuous spectrum. In the meantime, Shi with Yuan and Jian \cite{JS19, SY19} obtained the similar results for the following extended Harper's model
\begin{equation}\label{jacobiequ1}
\left(H(x,\omega)\phi\right)(n)=-a\left(x+(n+1)\omega\right)\phi(n+1)-\overline{a(x+n\omega)}\phi(n-1)+V(x+n\omega)\phi(n)
,\ n\in\mathbb{Z},\end{equation}
where
\[a(x)=\lambda_1\exp\left(2\pi i\left(x+\frac{\omega}{2}\right)\right)+\lambda_2+\lambda_3\exp\left(-2\pi i\left(x+\frac{\omega}{2}\right)\right), V(x)=\cos 2\pi x.\]
Recently, Xu-Zhao \cite{XZ20} gave another homogeneity result for the  non-critical extended Harper's with the Diophantine frequency $\cD_{c,A}$, i.e. for some $A>1$,
 \begin{equation}\label{dc}
\|n\omega\| \geq \frac{c}{|n|^{A}}\ \ \mbox{for
all}\ n\not=0.\end{equation}

After reviewing these results, we will find that except the first one \cite{GDSV}, the others are all about the operators who are in the subcritical regime or whose potential is the cosine function.
There is also another observation  that they all need the potential to be  analytic.  Cai-Wang \cite{CW21} announced a new breakthrough lately that for any $d-$dimension Diophantine frequency $\mathcal{D}_{d,c,A}$, i.e. for some $A>d$,
 \begin{equation}\label{ddc}
\|n\omega\| \geq \frac{c}{|n|^{A}}\ \ \mbox{for
all}\ n\not=0, \end{equation} there exists a $k_0=D_0A$, where $D_0$ is a numerical constant,   such that for any potential  $V\in C^k(\TT^d,\mathbb{R})$ with $k\ge k_0$,
the spectrum $\cS_{\omega}$ is homogeneous if $\|V\|_k\le \epsilon_0$, where $\epsilon_0=\epsilon_0(c,A,k,d)$. Note that  this operator is always having purely
absolutely continuous spectrum and its Lyapunov exponent is always zero. This is a traditional territory of reducibility in methodology, which can help people seek out the edge points of the spectral gaps,
where the cocycles are reducible to constant parabolic cocycles. Unfortunately, it can not work in the supercritical regime.

In this paper, we  consider the Schr\"odinger operators (\ref{jacobiequ}) with some non-analytic potential in the supercritical regime. We say a $C^{\infty}$ function $V$ is a Gevrey class $\mathcal{G}^s(\mathbb{T})$ for some $s>1$, if
\begin{equation}
  \label{div-con}
  \sup_{x\in\mbt}|\partial^m V(x)|\le K^m\left(m!\right)^s\ \ \ \forall m\ge 0
\end{equation}
for some constant $K\ge 0$.  Let
$$
\|V\|_{s,K}:=\frac{1}{3}\sup\limits_{m}\frac{(1+|m|)^2}{K^m(m!)^s}\|\partial^m V\|_{C^0(\mathbb{T})},
$$
$G^{s,K}(\mathbb{S}^1)=\{V\in C^\infty(\mathbb{T},\R)\||V\|_{s,K}<\infty\}$ and $G^s(\mathbb{T})=\cup_{K>0}G^{s,K}(\mathbb{T})$. Obviously, $G^1(\mathbb{T})$  is the space of analytic functions and for any $s\ge 1$, $G^s(\mathbb{T})$ is a subspace of the space of smooth functions.
\begin{theorem}\label{main-thm}
Consider the  quasi-periodic Schr\"odinger operators (\ref{jacobiequ}). Let $\omega\in\cD_{c,A}$, $V=\lambda v$ and $v\in \cG^{s,K}$ be a non-degenerate Gevrey function, which means it satisfies  the following called non-degeneracy condition
\begin{equation}
  \label{non-deg}
  \forall x\in \mathbb{T},\  \exists m,\ s.t.\  v^{(m)}(x)\not=0.
\end{equation}Then, there exists a constant $\lambda_0(v,c,A,s,K)$ such that for any $\lambda>\lambda_0$, the spectrum $\cS_\omega$ is homogeneous.
\end{theorem}
\begin{remark}
  So, the spectrum $\cS_\omega$ always has the positive measure in this condition.
\end{remark}
\begin{remark}
 It is obvious that all analytic functions on $\TT$ are non-degenerate. This non-degeneracy condition was first introduced by Klein \cite{K05} for the Gevrey potential. With its help, he proved the Anderson Localization,  the positivity and
 weak H\"older continuity of the Lyapunov exponent for the SDC frequency (\ref{sdc}). In our paper, we can extend them to  the Diophantine one (\ref{dc}) easily (see Remark \ref{AL}).
\end{remark}

As a contrast, it is not an easily thing to study the spectrum of the quasi-periodic  Schr\"odinger operators for non-SDC frequency in the supercritical regime. Actually, Goldstein-Schlag-Voda \cite{GSV16} attempted to proved the homogeneous spectrum of the multi-frequency  Schr\"odinger operators with more generic frequency, but they only succeed on a non full-measured subset of $\mathcal{D}_{d,c,A}$ (\ref{ddc}).
What's more, Avila-Last-Shamis-Zhou \cite{ALSZ} showed that even for the AMO (\ref{amo}), the spectrum
is not homogeneous if $\exp\left(-\frac{2}{3}\beta(\omega)\right)<\lambda<\exp\left(\frac{2}{3}\beta(\omega)\right)$.

The main reason for this situation is that it is very hard to obtain the good enough properties of the finite-volume determinant $f_n(x,E,\omega)$ with more generic frequency. For example, we need  it  to be not too small, since it is the denominator of the Green function. If the potential is the cosine,  it can be handled explicitly via the Lagrange interpolation for the trigonometric polynomial. But for general potential $V$, it is hard to judge  when it vanishes or not. In \cite{GS08},  Goldstein and Schlag used the subharmonicity, which comes from the analyticity of the potential, and Hilbert transform to obtain the BMO norm of $f_n$, and then applied the John-Nirenberg inequality to yield the following called large deviation theorem(LDT for short) with the SDC $\omega$:
\begin{equation}\label{ldtings08}
 \mes\left\{ x\in\mathbb{T}:\,\left|\log\left|f_{n}\left(x\right)\right|-\left\langle \log\left|f_{n}\right|\right\rangle \right|>n\delta\right\} \le C\exp\left(-c\delta \left(\delta_{SDC}^{(n)}\right)^{-1}\right),
\end{equation}
where $\delta_{SDC}^{(n)}=\frac{C_V(\log n)^{\alpha+2}}{n}$.  We call $\delta_{SDC}^{(n)}$ the smallest deviation, since (\ref{ldtings08}) works only when $\delta\gg \delta_{SDC}^{(n)}$. It is determined mainly by the holomorphic width of the analytic potential and the frequency, and the latter plays a central role.  Generally speaking, it will be larger when the irrational  frequency is ``more rational".  In \cite{GT20}, we applied their methods and obtain the corresponding LDTs for the Brjuno-R\"ussmann  frequency,  which is a famous extension of the SDC number. However, there exist some difficulties to apply it in this paper. The first one is that our potential is not analytic and then $\log\left|f_{n}\left(x\right)\right|$ is no longer subharmonic. The second one is that $\left\langle \log\left|f_{n}\right|\right\rangle$ may be nonexistent. What's more, even though we ignored the difference in potential, that is, we assumed the potential is analytic,  and then obtained the LDT for the Diophantine frequency $\cD_{c,A}$ in this paper,  we will find that the smallest deviation $\delta_{DC}^{(n)}$  is too large for the original methods in \cite{GDSV}.

So now, let's  make an introduction of  this paper and explain briefly  how to solve these problems. Section 2 is the preliminaries, including Lyapunov exponent,  Gevrey function, Cartan's estimate and semialgebraic set theory. Especially in subsection 2.2, we will show that the Gevrey cocycles have an analytic truncation, which will allow us to use some lemmas, which were for the analytic function or subharmonic one. In Section 3, we construct an  induction to get the LDTs.  Specifically, we  lead into  the {\L}ojasiewicz-type inequality from the non-degeneracy condition (\ref{non-deg}) to obtain the initial step, and apply the Cartan's estimate, instead of the BMO norm and the John-Nirenberg inequality, to yield the induction step. Then in the LDTs,  $\left\langle\log\left|f_{n}\left(x\right)\right|\right\rangle$ can be replaced by $L_n$ , which is very close to $\log\lambda$. As mentioned above, this LDT is not good enough for the original methods. So in Section 4, we  optimize them and lead into the  covering form of LDT to get the desired Wegner's estimate, which estimates the probability that there exists an eigenvalue in some interval $(E-\epsilon, E+\epsilon)$. At last, after the further optimization of  the way to produce spectral segments, which are closed to $\cS_{\omega}$, of considerable size, we  obtain the stability of spectrum, and then  prove the homogeneity.

\section{Preliminaries }
\subsection{Lyapunov exponent}
 It is obvious that the characteristic equation  $H(x,\omega)\phi=E\phi$ can be expressed as
\begin{equation}
\left (\begin{array}{cc}
  \phi(n+1) \\ \phi(n) \\
\end{array} \right )=\left ( \begin{array}{cc}
  V(x+n\omega)-E & -1 \\
1& 0 \\
  \end{array}\right )\left (\begin{array}{cc}
  \phi(n) \\ \phi(n-1) \\
\end{array} \right ).
\end{equation}
 Define
\begin{equation}\label{eq:1.3}
  M(x,E,\omega):=\left ( \begin{array}{cc}
  V(x)-E & -1 \\
 1& 0 \\
  \end{array}\right )
\end{equation}
 and the n-step transfer matrix
 $$M_n(x,E,\omega):=\prod_{k=n}^1M(x+k\omega,E,\omega).$$
 By the Kingman's subadditive ergodic theorem, the Lyapunov exponent
\begin{equation}\label{le}
  L(E,\omega)=\lim\limits_{n\to\infty}L_n(E,\omega)=\inf\limits_{n\to\infty}L_n(E,\omega)\ge 0
\end{equation} always exists, where
\begin{equation}\label{lne}
L_n(E,\omega)=\frac{1}{n}\int_\TT\log\|M_n(x,E,\omega)\|dx.
\end{equation}

\subsection{Gevrey function and analytic truncation} From now on, we always assume $V=\lambda v$, where $v\in\cG^{s,K}$. It is showed  that the condition (\ref{div-con}) is equivalent to the following weak exponential decay of the Fourier coefficients of $v$:
\begin{equation}\label{fou-v}
  \left|\hat v(k)\right|\le \|v\|_{s,K}\exp\left(-\rho |k|^{\frac{1}{s}}\right),\ \ \ \forall k\in\mathbb{Z},
\end{equation}
where $\rho=\frac{1}{K}$.
For every positive integer $n$, consider the truncation of the Gevrey function $v$:
\[
  v_n(x):=\sum_{|k|\leq \bar n}\hat v(k)e^{ikx},
\]
where $\bar n=\deg v_n$ will be determined later. Note that $v_n(x)$ is an analytic function on $\mathbb{T}$ and $v_n(z)$ is a holomorphic on the strip $\TT_n:=\{z:|\Im z|\le \rho_n\}$, where
\[\rho_n=\frac{\rho}{2}\bar n^{\frac{1}{s}-1}.\]Indeed,
if $z=x+iy$ with $|y|<\rho_n$, then
\[
  |v_n(z)|\le\sum_{|k|\le \bar n}\left|\hat v(k)\right|e^{|k||y|}\le  \|v\|_{s,K}\sum_{k=0}^{\bar n}\exp\left(-\rho |k|^{\frac{1}{s}}\right)e^{k|y|}
  < \|v\|_{s,K}\sum_{k=0}^{\bar n}\exp\left(-\frac{\rho}{2} |k|^{\frac{1}{s}}\right)<C.
\]
It is also obvious that for any $n$,
\begin{equation}
  \label{dis-g-gn1}
  |v(x)-v_n(x)|\le  \|v\|_{s,K}\exp\left(-\frac{\rho}{2}\bar n^{\frac{1}{s}}\right).
\end{equation}We choose $\bar n=n^{2s}$ to make the error in (\ref{dis-g-gn1})  super exponentially small:
\begin{equation}
  \label{dis-g-gn}
  |v(x)-v_n(x)|\le  \|v\|_{s,K}\exp\left(-\frac{\rho}{2}n^2\right).
\end{equation}
Then, the width of holomorphicity of $v_n$ becomes:
\[
  \rho_n=\frac{\rho}{2}n^{-2(s-1)}.
\]

Define the  analytic matrix related to $M_n$  by replacing $v$ by $v_n$:
\begin{equation}\label{def-mnt}
 M^t_n(x,E,\omega)=\prod_{j=n}^1\left ( \begin{array}{cc}
 E-\lambda v_n(x+j\omega) &-1\\
 1& 0 \\
  \end{array}\right ),
\end{equation}
For fixed $E$ and $\omega$,  $M^t_n(x,E,\omega)$ has its complex analytic extension  $M^t_n(z,E,\omega)$ on the strip $\TT_n$. In this paper, we always assume  $E \in [-\lambda \|v\|_\infty-2,\lambda\|v\|_\infty+2]$, or our question will be trivial.
Now,
\[
1\leq u_n(x,E,\omega),u_n^t(x,E,\omega)\leq \log(2\lambda \|v\|_\infty+4).
\]
Therefore ,  there exists a constant $\lambda'_0(v)$ such that if $\lambda > \lambda'_0(v)$ , then for any $n\ge 1$,
\begin{equation}\label{def-S}
 u_n(x,E,\omega),u_n^t(x,E,\omega)\leq S(\lambda)=\log\lambda+(\log\lambda)^{\frac{1}{2}}.
\end{equation}
Let $H_{[a,b]}(x,\omega)$ be the restriction of $H(x,\omega)$ to
the interval $[a,b]$ with Dirichlet boundary conditions, $\phi(a-1)=\phi(b+1)=0$,  and we
denote the corresponding Dirichlet determinant by $ f_{[a,b]}(x,E,\omega):= \det
(H_{[a,b]}(x,\omega)-E) $. We use $ E^{[a,b]}_j(x,\omega) $, $
\psi_j^{[a,b]}(x,\omega) $ to denote the eigenpairs of $
H_{[a,b]}(x,\omega) $, with $ \psi_j^{[a,b]}(x,\omega) $ being $
\ell^2 $-normalized. One has
\[f_{[a,b]}(x,E,\omega)=f_{b-a+1}\left(x+(a-1)\omega,E,\omega\right),\]
where $f_n:=f_{[1,n]}$. Simple computations yield that
\begin{equation}\label{Mnaxdet}
 M_n(x,E,\omega)=\left ( \begin{array}{cc}
  f_n(x,E,\omega) & -f_{n-1}(x+\omega,E,\omega) \\
 f_{n-1}(x,E,\omega)& - f_{n-2}(x+\omega,E,\omega) \\
  \end{array}\right ),
\end{equation}
and similar communication also holds between $f_n^t$ and $M_n^t$, where $f^t_{n}(x,E,\omega):= \det
\left(H^t_{n}(x,\omega)-E\right)$ and $H^t_n(x,\omega)$ denotes the  restriction operators $H_n(x,\omega)$ with the potential $\lambda v_n$.

Due to the  subharmonicities of $u_n^t$ and $\log|f^t_{n}|$, we can use the following  propositions in our paper.
\begin{prop}[Theorem 1.4 in \cite{GT20}]\label{sbet}
 Let $u:\Omega\to \IR$ be a subharmonic function on
a domain $\Omega\subset\IC$ and $\omega\in \cD_{c,A}$. Suppose that $\partial \Omega$ consists
of finitely many piece-wise $C^1$ curves, $\TT_{\rho}\subset \Omega$ and $\sup_{z\in \TT_{\rho}}u(z)<S$. Then, there exist constants  $c=c(c,A)$ and $ C=C(c,A)$ such that for any positive $n$ and $\delta>\frac{CS}{\rho}n^{-\nu}$,
\begin{equation}
  \label{eq-sbet}
  \mes\left (\left \{x\in \mathbb{T}:\left|\sum_{k=1}^n u(x+k\omega)-n\langle u\rangle\right|>\delta n\right \}\right ) \leq  \exp\left(-\frac{c \rho \delta n}{S}\right),
\end{equation}
where $\nu=\frac{1}{2A}$.
\end{prop}
\begin{lemma}[Lemma 2.2 in \cite{GS08}]
\label{lem:riesz} Let $u:\Omega\to \IR$ be a subharmonic function on
a domain $\Omega\subset\IC$. Suppose that $\partial \Omega$ consists
of finitely many piece-wise $C^1$ curves. There exists a positive
measure $\mu$ on~$\Omega$ such that for any $\Omega_1\Subset \Omega$
(i.e., $\Omega_1$ is a compactly contained subregion of~$\Omega$),
\begin{equation}
\label{eq:rieszrep} u(z) = \int_{\Omega_1}
\log|z-\zeta|\,d\mu(\zeta) + h(z),
\end{equation}
where $h$ is harmonic on~$\Omega_1$ and $\mu$ is unique with this
property. Moreover, $\mu$ and $h$ satisfy the bounds \be
\mu(\Omega_1) &\le& C(\Omega,\Omega_1)\,(\sup_{\Omega} u - \sup_{\Omega_1} u), \label{21002} \\
\|h-\sup_{\Omega_1}u\|_{L^\infty(\Omega_2)} &\le&
C(\Omega,\Omega_1,\Omega_2)\,(\sup_{\Omega} u - \sup_{\Omega_1} u)
\label{21003} \ee for any $\Omega_2\Subset\Omega_1$.
\end{lemma}
\begin{lemma}[Lemma 2.3 in \cite{BGS01}]\label{bgsbmonorm}
  Suppose u is subharmonic on $\mathbb{T}_{\rho}$, with $\mu(\mathbb{T}_{\rho})+\sup_{z\in\mathbb{T}_{\rho}}h(z)\le S$ where $\mu(\mathbb{T}_{\rho})$ and $h(z)$ comes from Lemma \ref{lem:riesz}. Furthermore, assume that $u=u_0+u_1$, where
\begin{equation}
\|u_0-<u_0>\|_{L^{\infty}(\mathbb{T})}\leq \epsilon_0\ \ \mbox{and}\ \ \|u_1\|_{L^1(\mathbb{T})}\leq \epsilon_1.
\end{equation}
Then for some constant $C_{\rho}$ depending only on $\rho$,
\[\|u\|_{BMO(\mathbb{T})}\leq C_{\rho}\left (\epsilon_0\log \left(\frac{S}{\epsilon_1}\right )+\sqrt{S\epsilon_1}\right).\]
\end{lemma}
\begin{lemma}[John-Nirenberg inequality]\label{jn-ineq}
  Let $f$ be a function of bounded mean oscillation on $\mathbb{T}$. Then there exist the absolute constants $C$ and $c$ such that for any $\gamma>0$
  \begin{equation}\label{jn}
    meas\left\{x\in\mathbb{T}: | f(x)-<f>|>\gamma \right\}\leq C\exp \left (-\frac{c\gamma}{\| f\|_{BMO}}\right ).
  \end{equation}
\end{lemma}~\\

At last,  we consider the distances between $u_n^t$ and $u_n$, and $f_n$ and $f_n^t$.
\begin{lemma}\label{dif-dis}
  Let $A,B:\mathbb{T}\to SL(2,\mathbb{R})$, and $a_{ij}(x)$ and $b_{ij}(x)$ be the $(i,j)$ elements of $\prod_{j=0}^{n-1}A(x+j\omega)$ and $\prod_{j=0}^{n-1}B(x+j\omega)$ respectively.  Assume $\max_{x\in\mathbb{T}}\left\{\|A(x)\|,\|B(x)\|\right\}\le \exp(S)$ and $\max_{x\in\mathbb{T}}\|A(x)-B(x)\|<\kappa$. Then, for any $x\in\mathbb{T}$, $n\in \mathbb{Z}^+$ and irrational  $\omega$,
  \begin{equation}\label{dis-dif}
   \left|a_{ij}(x)-b_{ij}(x)\right|\le  \left\|\prod_{j=0}^{n-1}A(x+j\omega)-\prod_{j=0}^{n-1}B(x+j\omega)\right\|\le n\kappa \exp\left((n-1)S\right),
  \end{equation}
  \begin{equation}\label{dis-dif-log}
    \left|\frac{1}{n}\log\left\|\prod_{j=0}^{n-1}A(x+j\omega)\right\|-\frac{1}{n}\log\left\|\prod_{j=0}^{n-1}B(x+j\omega)\right\|\right|\le \kappa \exp((n-1)S)
  \end{equation}
  and
  \begin{equation}\label{dis-dif-eme}
    \left|\frac{1}{n}\log\left|a_{ij}(x)\right|-\frac{1}{n}\log\left|b_{ij}(x)\right|\right|\le \frac{\kappa \exp((n-1)S)}{\max\left\{\left|a_{ij}(x)\right|, \left|b_{ij}(x)\right| \right\}},
  \end{equation}
  provided the right-hand sides of (\ref{dis-dif-log}) and (\ref{dis-dif-eme}) are less than $1/2.$ Moreover, the quantity $\exp((n-1)S$ in these expressions can be replaced by any new bound of $\left\|\prod_{j=0}^{n-1}A(x+j\omega)\right\|$ and $\left\|\prod_{j=0}^{n-1}B(x+j\omega)\right\|$.
\end{lemma}
\begin{proof}
\begin{eqnarray*}
 \left\|\prod_{j=0}^{n-1}A(x+j\omega)-\prod_{j=0}^{n-1}B(x+j\omega)\right\|&\le&\left\|\prod_{j=0}^{n-1}A(x+j\omega)-\prod_{j=0}^{n-2}A(x+j\omega)B(x+(n-1)\omega)\right\|\\
&&\ \ +\cdots\\
&&\ \ +\left\|A(x)A(x+\omega)\prod_{j=2}^{n-1}B(x+j\omega)-A(x)\prod_{j=1}^{n-1}B(x+j\omega)\right\|\\
&&\ \ +\left\|A(x)\prod_{j=1}^{n-1}B(x+j\omega)-\prod_{j=0}^{n-1}B(x+j\omega)\right\|\\
&\le& n\kappa \exp((n-1)S).
\end{eqnarray*}
Therefore,
\allowdisplaybreaks
\begin{eqnarray*}
  &&\left|\frac{1}{n}\log\left\|\prod_{j=0}^{n-1}A(x+j\omega)\right\|-\frac{1}{n}\log\left\|\prod_{j=0}^{n-1}B(x+j\omega)\right\|\right|=\frac{1}{n}\left|\log\frac{\left\|\prod_{j=0}^{n-1}A(x+j\omega)\right\|}{\left\|\prod_{j=0}^{n-1}B(x+j\omega)\right\|}\right|\\
  &=&\frac{1}{n}\left|\log\left(1+\frac{\left\|\prod_{j=0}^{n-1}A(x+j\omega)\right\|
  -\left\|\prod_{j=0}^{n-1}B(x+j\omega)\right\|}{\left\|\prod_{j=0}^{n-1}B(x+j\omega)\right\|}\right)\right|\nonumber\\
  &\leq &\frac{1}{n}\left|\frac{\left\|\prod_{j=0}^{n-1}A(x+j\omega)\right\|
  -\left\|\prod_{j=0}^{n-1}B(x+j\omega)\right\|}{\left\|\prod_{j=0}^{n-1}B(x+j\omega)\right\|}\right|\\
  &\le &\kappa \exp((n-1)S).
\end{eqnarray*}
And (\ref{dis-dif-eme}) can be obtained similarly.
\end{proof}
Due to Lemma \ref{dif-dis} and (\ref{dis-g-gn}), it has that
\[
\left|\|M_n^t(x,E,\omega)\|-\|M_n(x,E,\omega)\|\right|\leq n \|v\|_{s,K}\exp((n-1)S)\exp\left(-\frac{\rho}{2}n^2\right)
\]
Therefore, for any $n>n_0=(\log \lambda)^{\tilde C}$, where $\tilde C:=\tilde C(s,K, \|v\|_{s,K})$,
\begin{equation}\label{dis-un-unt}
|u_n(x,E,\omega)-u_n^t(x,E,\omega)|\leq e^{-n}, \forall x\in \TT,
\end{equation}
and
\begin{equation}\label{dis-ln-lnt}
 |L_n(E,\omega)-L_n^t(E,\omega)|<e^{-n}.
\end{equation}

\subsection{Cartan's Estimate}\label{sec:Cartan} We recall the
definition of Cartan sets from \cite{GS08}. We use the notation $
\cD(z_0,r)=\{ z\in \C: |z-z_0|<r \} $.
\begin{defi}\label{defi:cartansets} Let $H \ge 1$.  For an arbitrary set $\cB \subset \cD(z_0,
1)\subset \C$ we say that $\cB \in \car_1(H, J)$ if $\cB\subset
\bigcup\limits^{j_0}_{j=1} \cD(z_j, r_j)$ with $j_0 \le J$, and
\begin{equation}
\sum_j\, r_j < e^{-H}\ .
\end{equation}
If $d\ge 1$ is an  integer and $\cB \subset
\prod\limits_{j=1}^d \cD(z_{j,0}, 1)\subset \C^d$, then we define
inductively that $\cB\in \car_d(H, J)$ if for any $1 \le j \le d$ there
exists $\cB_j \subset \cD(z_{j,0}, 1)\subset \IC, \cB_j \in \car_1(H,
J)$ so that $\cB_z^{(j)} \in \car_{d-1}(H, J)$ for any $z \in \IC
\setminus \cB_j$,  here $\cB_z^{(j)} = \left\{(z_1, \dots, z_d) \in \cB:
z_j = z\right\}$.
\end{defi}

The definition is motivated by the following
generalization of the usual Cartan estimate to several variables. Note
that given a set $ S $ that has a centre of symmetry, we will let
$ \alpha S $, $ \alpha>0 $, stand for the set scaled with respect to
its centre of symmetry.

\begin{lemma}[{\cite[Lem.~2.15]{GS08}}]\label{lem:high_cart}
 Let $\varphi(z_1, \dots, z_d)$ be an analytic function defined
on a polydisk $\cP = \prod\limits^d_{j=1} \cD(z_{j,0}, 1)$, $z_{j,0} \in
\IC$.  Let $M \ge \sup\limits_{z\in\cP} \log |\varphi(z)|$,  $m \le
\log|\varphi(z_0) |$,
$z_0 = (z_{1,0},\dots, z_{d,0})$.  Given $H
\gg 1$ there exists a set $\cB \subset \cP$,  $\cB \in
\car_d\left(H^{1/d}, J\right)$, $J = C_d H(M - m)$,  such that
\begin{equation}\label{eq:cart_bd}
  \log | \varphi(z)| > M-C_d H(M-m)
\end{equation}
for any $z \in \frac{1}{6}\cP\setminus \cB$.
Furthermore, when $ d=1 $ we can take
$ J=C(M-m) $ and keep only the disks of $ \cB $ containing a zero of $\phi$ in them.
\end{lemma}

\subsection{semialgebraic set}\label{sec:semi-set}
A set $ \cS\subset\R^d $ is called semialgebraic if it is
a finite union of sets defined by a finite number of polynomial
equalities and inequalities. More precisely, a semialgebraic set
$ \cS\subset\R^d $ is given by an expression
\begin{equation*}
  \cS=\cup_j \cap_{l\in L_j} \{ P_l k_{jl} 0  \},
\end{equation*}
where $ \{P_1,\ldots,P_k\} $ is a collection of polynomials of $ d $
variables,
\begin{equation*}
  L_j\subset\{1,\ldots,k\}\text{ and }k_{jl}\in\{>,<,=\}.
\end{equation*}
If the degrees of the polynomials are
bounded by $ p $, then we say that the degree of $ \cS $ is bounded by
$ kp $. See \cite{B05} for more information on
semialgebraic sets.

In our context, semialgebraic sets can be introduced by approximating
the Gevrey function $ v$ with a polynomial $ \tilde v $. More
precisely, given $ N\ge 1 $, by truncating $ v$'s Fourier series and
the Taylor series of the trigonometric functions, one can obtain a
polynomial $ \tilde v_n $ of degree less than
\begin{equation*}
  C(d,\rho)(1+\log \|v\|_\infty)n^{4s}
\end{equation*}
such that
\begin{equation}\label{eq:V-tilde}
  \sup_{x\in\TT^d}|v(x)-\tilde v_n(x)|\le \exp(-n^2).
\end{equation}
We also need the following lemma related to this theory.
\begin{corollary}[Theorem 9.3 in \cite{B05}]\label{number}
Let $ \cS\subset [0,1]^d $ be semialgebraic of degree $ p $.  Then, the number of connected components of $\cS$ does not exceed $k^d\left(O(p)\right)^d$.
\end{corollary}

\section{Large Deviation Theorems }
In this section, we will apply the  induction to obtain the  following LDTs:
\begin{prop}\label{ldt-thm}
Let  $\omega\in\cD(c,A)$ and  $\lambda>\lambda_0(v,c,A,s,K)$. Then,
\begin{equation}\label{posi-le-g}
L(E,\omega)\gtrsim \log\lambda-(\log\lambda)^{\frac{1}{2}}.
\end{equation}
And there exists a small constant $\nu(A)$  such that  for any   $n>n_0(\lambda,c,A,v,s,K)$ and $\delta\gg n^{-\nu}$
\begin{equation}\label{ldt-ung}
   meas\left(\left\{x:\left|u_{n}(x,E,\omega)-L_{n}(E,\omega)\right|>\delta
\right\}\right)<\exp\left(-c\delta n^{\nu}\right),
  \end{equation}
\begin{equation}\label{ldt-fng}meas\left(\left\{x:\left|\frac{1}{n}\log|f_{n}(x,E,\omega)|-L_{n}(E,\omega)\right|>\delta
\right\}\right)<\exp\left(-c\delta n^{\nu}\right).
  \end{equation}
  Moreover, the set on the left-hand side of  (\ref{ldt-fng}) is  contained in the union of less than  $n^{C_0s}$ intervals, where $C_0=C_0(v)$ is a constant.
\end{prop}
\begin{remark}\label{AL}
We have obtained the positive Lyapunov exponent in (\ref{posi-le-g}). The standard methods for the continuity and Anderson  Localization can be seen in many references, such as \cite{B05,K05}. The readers will find that the key of these methods is the ldt(\ref{ldt-ung}) and the other part is trivial.
\end{remark}

\subsection{The initial step}
In this part, we first list the following lemma, which is called avalanche principle and  will be applied several times in this paper.
\begin{prop}[Proposition 2.2 in \cite{GS01}]
\label{prop:AP} Let $A_1,\ldots,A_n$ be a sequence of  $2\times
2$--matrices whose determinants satisfy
\begin{equation}
\label{eq:detsmall} \max\limits_{1\le j\le n}|\det A_j|\le 1.
\end{equation}
Suppose that \be
&&\min_{1\le j\le n}\|A_j\|\ge\mu>n\mbox{\ \ \ and}\label{large}\\
   &&\max_{1\le j<n}[\log\|A_{j+1}\|+\log\|A_j\|-\log\|A_{j+1}A_{j}\|]<\frac12\log\mu\label{diff}.
\ee Then
\begin{equation}
\Bigl|\log\|A_n\cdot\ldots\cdot A_1\|+\sum_{j=2}^{n-1}
\log\|A_j\|-\sum_{j=1}^{n-1}\log\|A_{j+1}A_{j}\|\Bigr| <
C\frac{n}{\mu} \label{eq:AP}
\end{equation}
with some absolute constant $C$.
\end{prop}
We shall fix $\omega\in \cD_{c,A}$  so that it can be suppressed from the notations. Now, we present the initial step in the following lemma:
\begin{lemma} For $ \lambda\geq\lambda_0$ and $ n\leq\exp\left({(\log\lambda)}^{\frac{1}{4}}\right) $,  we have
\begin{equation}\label{ln0}
 |L_n(E)-\log\lambda|\lesssim(\log\lambda)^\frac{1}{2},
\end{equation}
and there exists a set $ \mathcal {B}_n $ , $ mes(\mathcal {B}_n) < \exp\left(-(\log\lambda)^\frac{1}{3}\right) $ , such that
\begin{equation}\label{ldt-mn-0}
  \left|\frac{1}{n}\log\|M_n(x,E)\| - \log\lambda\right| \lesssim (\log\lambda)^\frac{1}{2}
\end{equation}
and
\begin{equation}\label{ldt-fn-0}
  \left|\log\|M_n(x,E)\| - \log\|f_n(x,E)\|\right| \lesssim (\log\lambda)^\frac{1}{2}
\end{equation}
for any $ x \not\in \mathcal {B}_n$.
\end{lemma}
\begin{proof} Klein \cite{K05} proved that for the non-degenerate Gevrey function (\ref{non-deg}),  the following {\L}ojasiewicz-type inequality holds: there exists a constant $\alpha:=\alpha(v)$ such for any $\delta>0$ and any $\gamma$,
\[mes \left\{ x: \left|v(x) - \gamma\right| < \delta \right\} < \delta^{\alpha}.
\]
Choosing $ \delta = \exp\left(-(\log\lambda)^\frac{1}{2}\right) $ and $\gamma=\frac{E}{\lambda}$, we have that
\[ mes\left\{ x:\log|\lambda v-E|-\log\lambda < -(\log\lambda)^\frac{1}{2}\right\}\le \exp\left(-\alpha(\log\lambda)^\frac{1}{2}\right)
.\]
On the other hand, by (\ref{def-S}), it yields that  for any $\lambda>\lambda_0$,
\[\log{|\lambda v-E|}-\log\lambda < (\log \lambda )^\frac{1}{2},\ \forall x\in\TT.\]
Set $ \mathcal {B} = \left\{x: \left|\log{|\lambda v-E|}-\log\lambda \right|< (\log \lambda )^\frac{1}{2}\right\} $ and  $\mathcal {B}_n=\cup_{1\leq j \leq n}(\mathcal {B}-j\omega)$. Then
\[
mes(\mathcal {B}_n) \leq n\exp\left(-\alpha(\log\lambda)^\frac{1}{2}\right)< \exp\left(-(\log\lambda)^\frac{1}{3}\right).
\]
Now, for $ x \not\in \mathcal {B}_n $ , $ 1\leq j \leq n $,
\[
\left|\log|\lambda v(x+j\omega) - E| - \log\lambda\right| \leq (\log\lambda)^\frac{1}{2}
\]
and hence for $l=1,2$,
\begin{equation}\label{dis-1-2}
\left|\log\|f_l(x+(j-1)\omega,E)\| - l\log\lambda\right|, \left|\log\|M_l(x+(j-1)\omega,E)\| - l\log\lambda\right| \lesssim (\log\lambda)^\frac{1}{2}.
\end{equation}
Integrating it and recalling the setting $\lambda>\lambda_0$ and (\ref{def-S}), we obtained that
\begin{equation}\label{l12}
\left|L_1-\log\lambda\right|, |L_2-\log\lambda|\lesssim (\log\lambda)^{\frac{1}{2}}.
\end{equation}
Now, we can apply the avalanche principle, Proposition \ref{prop:AP}, and yield that  for any $ x \not\in \mathcal {B}_n $,
\begin{equation}\label{ap-mn0}
 \log\|M_n(x,E)\| = \sum_{j=0}^{n-2}\log\|M_2(x+j\omega,E)\| - \sum_{j=1}^{n-2}\log\|M_1(x+j\omega,E)\|+O(\lambda^{-\frac{1}{2}}),
\end{equation}
and
  \begin{multline}\label{eq:fN-AP}
	    \log|f_n(x,E)|\\
    =\log\left\|M_2(x,E)\begin{bmatrix}
        1 & 0\\
        0 & 0
      \end{bmatrix}\right\|+
    \sum_{j=1}^{n-3} \log\left \| M_{2} (x+j\omega,E) \right\|+\log\left\|\begin{bmatrix}
        1 & 0\\
        0 & 0
      \end{bmatrix}M_2(x+(n-2)\omega,E)\right\|\\
    - \sum_{j=1}^{n-2} \log \left\| M_1 (x+j\omega,E)\right \| + O(\lambda^{-\frac{1}{2}}).
  \end{multline}
  Here we used the fact that
  \begin{equation}\label{eq:fN-MN}
	f_n(x,E)=\begin{bmatrix}
        1 & 0\\
        0 & 0
      \end{bmatrix}M_n(x,E)\begin{bmatrix}
        1 & 0\\
        0 & 0
      \end{bmatrix}.
  \end{equation}
It follows that (\ref{ldt-mn-0}) and (\ref{ldt-fn-0})   hold  by (\ref{dis-1-2}).
Integrating (\ref{ap-mn0}) yields
\[
|nL_n(E) - (n-1)2L_2(E) + (n-2)L_1(E)|\leq C\lambda^{-\frac{1}{2}} + 4mes(\mathcal {B}_n)S \leq \exp\left(-(\log\lambda)^{\frac{1}{4}}\right)
\]
Combining it with (\ref{l12}), we obtain (\ref{ln0}).
\end{proof}
The following corollary shows that we can obtain the new upper bound of $M_n$, when the LDT holds for the subharmonic function $M_n^t$:
\begin{corollary}\label{cor:logupper} Assume the following LDT holds  for some $n>n_0$ and  any  fixed $|y| \le n^{-2s}$
\begin{equation}\label{mes}
\mes\left\{x\in \TT:\left|\log\|M^t_n(x+iy,E)\|-nL_n^t(y, E)\right|>n\delta\right\}<\exp\left(-c\delta n^{\frac{\nu}{4}}\right),
\end{equation}
where
\[L_n^t(y,E)=\frac{1}{n}\int_{\mathbb{T}}\log\|M^t_n(x+iy,E)\|dx.\]
 Then,
\begin{equation}\label{sup-bound-mn-wd}
\sup_{x\in \TT}\log |f_n(x,E)|\le \sup_{x\in \TT}\log\|M_n(x,E)\|\leq n L_n(E)+n^{1-\frac{\nu}{10}}.\end{equation}
\end{corollary}
\begin{proof}
Lemma 4.1 in \cite{GS08} proved that  for any $|y_{1}|,|y_{2}|\le \rho_n$ we have
\begin{equation}\label{L(r1)-L(r2)}
  \left|L^t_n(y_1,E)-L^t_n(y_2,E)\right|\leq \frac{S}{\rho_n}|y_1-y_2|.
\end{equation}
Combining it with (\ref{dis-ln-lnt}), we obtain that for any $|y|<n^{-4s}$,
\begin{equation}\label{ldt-complex-ln-real}
\mes\left\{x\in \TT:\left|\log\|M^t_n(x+iy,E)\|-nL_n(E)\right|>n\delta+n^{-2s}\right\}<\exp\left(-c\delta n^{\frac{\nu}{4}}\right).
\end{equation}
On the other hand,  due to the sub-mean value property for subharmonic functions, we have for any $|y|<n^{-4s}$ and $r>0$,
\begin{equation}\label{leq}
\log\|M^t_n(x+iy,E)\|\leq \frac{1}{\pi r^2}\int_{D(x+iy,r)}\log\|M^t_n(z,E)\|dz.
\end{equation}
Denote by $\cB_{y}\subset \TT$ the set in (\ref{ldt-complex-ln-real}) and choose $r=n^{-4s}<\rho_n$. Let
\[\cB=\{z=x+iy\in[0,1]\times(-r,r):x\in \cB_y\}.\]
Due to (\ref{ldt-complex-ln-real}) with choosing $\delta=n^{-\frac{\nu}{8}}$,  we have
\begin{equation}\label{pi}
\frac{1}{\pi r^2}\int_{D(x,r)\backslash B}\log\|M^t_n(z,E)\|dz\leq nL_n(E)+4n^{1-\frac{\nu}{8}}.
\end{equation}
By  the H\"older inequality, we have that
\begin{equation}\label{int-bad-for-upp-bound}
\frac{1}{\pi r^2}\int_{D(x,r)\bigcap \cB} \log\|M^t_n(z,E)\|d\xi d\zeta\ll n^{1-\frac{\nu}{8}}.
\end{equation}
Combining (\ref{leq}), (\ref{pi}) and (\ref{int-bad-for-upp-bound}), we obtain
\begin{equation}\label{sup-bound-mnt}\sup_{x\in \TT,|y|\le n^{-4s}}\log\|M^t_n(x+iy,E)\|\leq nL(E)+5n^{1-\frac{\nu}{8}}.\end{equation}
Hence, we prove this lemma by the relationship (\ref{dis-un-unt}).
\end{proof}

\subsection{The inductive step}
\begin{lemma}\label{ind-step} Let $\lambda\geq\lambda_0$.
 Assume for any $l\in [n_i,4n_i]$, where $n_i\ge n_0$, it yields that
  \begin{equation}\label{low-up-bound-le-ni}
  \left|L_{l}(E)- \log\lambda\right|\lesssim (\log\lambda)^\frac{1}{2} ,
  \end{equation}
  \begin{equation}\label{ldt-ni}
    meas\left\{x:\left|u_{l}(x,E)-L_{l}(E)\right|>\frac{1}{100}\log\lambda
\right\}<\exp\left(-n^{\frac{\nu}{3}}\right),
  \end{equation}
  \begin{equation}\label{ldt-fn-ni}
  \mes\left\{ x\in\mathbb{T}:\,\left|\frac{1}{l}\log\left\Vert f_{l}\left(x,E\right)\right\Vert-L_l(E) \right|>\frac{1}{100}\log\lambda\right\} \le \exp\left(-n^{\frac{\nu}{20}}\right).
\end{equation}
Then for any $n\in \left[ n_i^{2B},n_i^{5B^2}\right]$, where $B=s\nu^{-1}$,  we have that
\begin{equation}\label{low-up-bound-le-ni1}
     \left|L_{n}(E)-\log\lambda\right|\lesssim (\log\lambda)^\frac{1}{2},
  \end{equation}
  \begin{equation}\label{ldt-n-i1}
   meas\left(\left\{x:\left|u_{n}(x,E)-L_{n}(E)\right|>\delta
\right\}\right)<\exp\left(-c\delta n^\frac{\nu}{2}\right),
  \end{equation}
\begin{equation}\label{ldt-fn-ni1}
  \mes\left\{ x\in\mathbb{T}:\,\left|\frac{1}{n}\log\left\Vert f_{n}\left(x,E\right)\right\Vert-L_n(E) \right|>\delta \right\} \le \exp\left(-c\delta n^{\frac{\nu}{15}}\right).
\end{equation}
\end{lemma}
\begin{proof}We suppress $E$ from the notations for ease in this proof.
 Let $l=n_i$  and $n\in \left[ n_i^{2B},n_i^{5B^2}\right]$. Then,  $n=l+\left(m-2\right)l+l'$ with $2l\le l'\le3l$.
Set $A_{j}\left(x\right)=M_{l}\left(x+\left(j-1\right)l\omega\right)$, $j=1,\ldots,m-1$, and $A_{m}\left(x\right)=M_{l'}$. Then, due to (\ref{ldt-ni}), there exists a set  $\cG_i$ satisfying
$$meas(\mbt \backslash \cG_i)\le 4(m+2)\cdot \exp\left(-l^\frac{\nu}{3}\right)\le\exp\left(-l^\frac{\nu}{4}\right)$$
such that for any $x\in\cG_i$,
\[\|A_j(x)\|>\exp\left(\frac{9}{10}l\log\lambda\right),\ \|A_m(x)\|>\exp\left(\frac{9}{10}l'\log\lambda\right),\]
\[\big |\log\|A_j(x)\|+\log\|A_{j+1}(x)\|-\log\|A_{j+1}(x)A_j(x)\|\big |\leq \frac{3}{25}l\log\lambda,\]
and
\[\big |\log\|A_{m-1}(x)\|+\log\|A_m(x)\|-\log\|A_{m}(x)A_{m-1}\|\big |\leq \frac{8}{25}l\log\lambda.\]
Now the hypothesis of Avalanche Principle
are satisfied and hence
\begin{equation}\label{apforunu}
\log\left\Vert M_{n}\left(x\right)\right\Vert +\sum_{j=2}^{m-1}\log\left\Vert A_{j}\left(x\right)\right\Vert -\sum_{j=1}^{m-1}\log\left\Vert A_{j+1}\left(x\right)A_{j}\left(x\right)\right\Vert =O\left(\frac{1}{l}\right)
\end{equation}
up to a set of measure less than $\exp\left(-l^\frac{\nu}{4}\right)$.
If we set
\[
u_{0}\left(x\right)=\log\left\Vert A_{m}\left(x\right)A_{m-1}\left(x\right)\right\Vert +\log\left\Vert A_{2}\left(x\right)A_{1}\left(x\right)\right\Vert,
\]
then  the previous relation can be rewritten as
\[
\log\left\Vert M_{n}\left(x\right)\right\Vert +\sum_{j=2}^{m-1}\log\left\Vert M_{l}\left(x+\left(j-1\right)l\omega\right)\right\Vert \\
-\sum_{j=2}^{m-2}\log\left\Vert M_{2l}\left(x+\left(j-1\right)l\omega\right)\right\Vert -u_{0}\left(x\right)=O\left(\frac{1}{l}\right).
\]
Similarly, for any $0\le k<l-1$,
\[
\log\left\Vert M_{n}\left(x\right)\right\Vert +\sum_{j=2}^{m-1}\log\left\Vert M_{l}\left(x+k\omega+\left(j-1\right)l\omega\right)\right\Vert \\
-\sum_{j=2}^{m-2}\log\left\Vert M_{2l}\left(x+k\omega+\left(j-1\right)l\omega\right)\right\Vert -u_{k}\left(x\right)=O\left(\frac{1}{l}\right),
\]
where
\[u_k(x)=\log\left\| M_{l'-k}\left(x+k\omega+(m-1)l\omega\right)\cdot A_{m-1}(x+k\omega)  \right \|+\log\left\|A_2(x+k\omega)\cdot M_{l+k}\left(x\right)\right\| ,\]
which means that we
decrease the length of $A_{m}$ by $k$ and increase the length of
$A_{1}$ by $k$. Adding these equations and dividing by $l$
yield
\[
\log\left\Vert M_{n}\left(x\right)\right\Vert +\sum_{j=l}^{\left(m-1\right)l-1}\frac{1}{l}\log\left\Vert M_{l}\left(x+j\omega\right)\right\Vert -\sum_{j=l}^{\left(m-2\right)l-1}\frac{1}{l}\log\left\Vert M_{2l}\left(x+j\omega\right)\right\Vert
-\sum_{k=0}^{l-1}\frac{1}{l}u_{k}\left(x\right)=O\left(\frac{1}{l}\right)
\]
up to a set of measure less than $\exp\left(-l^\frac{\nu}{5}\right)$. What we have done is to obtain the Dirichlet sums,  $\displaystyle \sum_{j=l}^{\left(m-1\right)l-1}\frac{1}{l}\log\left\Vert M_{l}\left(x+j\omega\right)\right\Vert$ and $\displaystyle \sum_{j=l}^{\left(m-2\right)l-1}\frac{1}{l}\log\left\Vert M_{2l}\left(x+j\omega\right)\right\Vert$.
 However, Proposition \ref{sbet} can not be applied to them, since them are not subharmonic. So, we need to change them by the subharmonic truncation functions, $\displaystyle \sum_{j=l}^{\left(m-1\right)l-1}\frac{1}{l}\log\left\Vert M_{l}^{t}\left(x+j\omega\right)\right\Vert$ and $\displaystyle \sum_{j=l}^{\left(m-2\right)l-1}\frac{1}{l}\log\left\Vert M_{2l}^{t}\left(x+j\omega\right)\right\Vert$. Due to (\ref{dis-un-unt}),
\[\left|\sum_{j=l}^{\left(m-1\right)l-1}\frac{1}{l}\log\left\Vert M_{l}\left(x+j\omega\right)\right\Vert-\sum_{j=l}^{\left(m-1\right)l-1}\frac{1}{l}\log\left\Vert M_{l}^{t}\left(x+j\omega\right)\right\Vert\right|<n\exp(-l)\ll \frac{1}{l},\]
\[ \left|\sum_{j=l}^{\left(m-2\right)l-1}\frac{1}{l}\log\left\Vert M_{2l}\left(x+j\omega\right)\right\Vert-\sum_{j=l}^{\left(m-2\right)l-1}\frac{1}{l}\log\left\Vert M_{2l}^{t}\left(x+j\omega\right)\right\Vert\right|<n\exp(-l)\ll \frac{1}{l},\]
Obviously, the similar relationships between $u_k$ and $u_k^t$ also hold.
Hence,
\[
\log\left\Vert M_{n}\left(x\right)\right\Vert +\sum_{j=l}^{\left(m-1\right)l-1}\frac{1}{l}\log\left\Vert M_{l}^{t}\left(x+j\omega\right)\right\Vert -\sum_{j=l}^{\left(m-2\right)l-1}\frac{1}{l}\log\left\Vert M_{2l}^{t}\left(x+j\omega\right)\right\Vert
-\sum_{k=0}^{l-1}\frac{1}{l}u_{k}^t\left(x\right)=O\left(\frac{1}{l}\right)
\]
up to this set.
Recall that $\frac{1}{l}\log\left\Vert M_{l}^{t}\left(x\right)\right\Vert$ is a subharmonic function on $\TT_l$ with the maximum $S$ and $ml\sim n$. So, Proposition \ref{sbet} can be applied with the smallest deviation $\frac{3CS}{\rho}l^{2(s-1)}n^{-\nu}$ and we obtain that
\[
\sum_{j=l}^{\left(m-1\right)l-1}\frac{1}{l}\log\left\Vert M_{l}^{t}\left(x+j\omega\right)\right\Vert -\sum_{j=l}^{\left(m-2\right)l-1}\frac{1}{l}\log\left\Vert M_{2l}^{t}\left(x+j\omega\right)\right\Vert
=\left(m-2\right)lL_{l}^{t}-\left(m-3\right)lL_{2l}^{t}+O\left(l^{2s}n^{1-\nu}\right)
\]
up to a set of measure less than $2\exp\left(-cn^{1-\nu}\right)$. Note that $u^t_{k}$, $k=0,\ldots,l-1$ also have the subharmonic extensions and Proposition \ref{sbet} can be applied for these functions with $n=1$ and $\delta=l^{2s}n^{1-\nu}\gg 1$. So,
\[
\sum_{k=0}^{l-1}\frac{1}{l}u^t_{k}\left(x\right)-\sum_{k=0}^{l-1}\frac{1}{l}\left\langle u^t_{k}\right\rangle =O\left(l^{2s}n^{1-\nu}\right)
\]
up to a set of measure less than $l\exp\left(- cn^{1-\nu}\right)$. Thus, combining  these equations, we have that
\begin{equation}\label{M_impure_AP2}
\log\left\Vert M_{n}\left(x\right)\right\Vert+\left(m-2\right)lL_{l}^{t}-\left(m-3\right)lL_{2l}^{t}-\sum_{k=0}^{l-1}\frac{1}{l}\left\langle u^t_{k}\right\rangle =O\left(l^{2s}n^{1-\nu}\right)
\end{equation}
up to a set of measure less than $\exp\left(-l^\frac{\nu}{5}\right)+2\exp\left(-cn^{1-\nu}\right)+l\exp\left(-cn^{1-\nu}\right)$. Note that for $n>n_0$,
\begin{equation}\label{G-inequ}
  \exp\left(-l^\frac{\nu}{5}\right)+2\exp\left(-cn^{1-\nu}\right)+l\exp\left(- cn^{1-\nu}\right)\ll n^{-10s}\ \mbox{and}\ l^{2s}n^{1-\nu}\le n^{1-\frac{3\nu}{4}}
\end{equation}
Integrating (\ref{M_impure_AP2}) , yields
\begin{equation}\label{ap-nln}
  nL_n+\left(m-2\right)lL_{l}^{t}-\left(m-3\right)lL_{2l}^{t}-\sum_{k=0}^{l-1}\frac{1}{l}\left\langle u^t_{k}\right\rangle =O\left(n^{1-\frac{3\nu}{4}}\right).
\end{equation}
Hence, we obtain (\ref{low-up-bound-le-ni1}) by (\ref{ap-nln}) and (\ref{low-up-bound-le-ni}).

Combining (\ref{M_impure_AP2}),  (\ref{ap-nln}) and (\ref{dis-un-unt}), we have
\begin{equation}
\left|\log\left\Vert M^t_{n}\left(x\right)\right\Vert-nL^t_n \right|=O\left(n^{1-\frac{3\nu}{4}}\right)\label{f^t-avg-bound2}
\end{equation}
up to a set of measure less than $n^{-10s}$. Let $\cB$ be this exceptional set  and define
\[
\log\left\Vert M_{n}^{t}\left(x\right)\right\Vert-nL_n^t  =u_{0}+u_{1}
\]
 where $u_{0}=0$ on $\cB$ and $u_{1}=0$ on $\TT\setminus\cB$.
Obviously,  $\left\Vert u_{0}-\left\langle u_{0}\right\rangle \right\Vert _{L^{\infty}\left(\TT\right)}=O\left(n^{1-\frac{3\nu}{4}}\right)$
\[
\left\Vert u_{1}\right\Vert _{L^{2}\left(\TT\right)}\le n^{-3s}.
\]
By Lemma \ref{bgsbmonorm} and choosing $S=n^{2s}$ to make $\rho$ be uniform, we obtain that
\[
\left\Vert\log\left\Vert M_{n}^{t}\left(x\right)\right\Vert\right\Vert _{BMO\left(\TT\right)}=O(n^{1-\frac{\nu}{2}}).
\]
Thus, Lemma \ref{jn-ineq} implies us that for any $n>n_0$ and any $\delta>0$,
\begin{equation}\label{ind-ldt-fnt}\mes\left\{ x\in\mathbb{T}:\,\left|\log\left\Vert M_{n}^{t}\left(x\right)\right\Vert-nL_n^t \right|>n\delta\right\} \le C\exp\left(-c\delta n^{\frac{\nu}{2}}\right).
\end{equation}
Hence, we get (\ref{ldt-n-i1}) by (\ref{dis-un-unt}).

Like (\ref{eq:fN-AP}), due to (\ref{ldt-fn-ni}) and (\ref{ldt-ni}), we can also apply the avalanche principle to expand $|f_n(x,E)|$. Combining this with (\ref{apforunu}), we obtain that
 \begin{multline}\label{eq4.4f}
    \log |f_n(x)|=\log \| M_n (x) \|
    + \log\left\|M_{2\ell}(x)\begin{bmatrix} 1 & 0\\ 0 &
        0\end{bmatrix}\right\|-\log \|M_{2\ell}(x)\|\\
    +\log\left\|\begin{bmatrix} 1 & 0\\ 0 &
        0\end{bmatrix}M_{l+l'}(x+(m-2)l\omega)\right\|
      -\log\|M_{l+l'}(x+(m-2)l\omega)\|
    + O(1/l)\\
    \ge \log \| M_n (x) \|-l'\log\lambda \ge
    nL_n-n^{1-\frac{\nu}{4}}
  \end{multline}
 up to a set of measure less than $\exp\left(-l^\frac{\nu}{40}\right)$. In particular,
  for any $x_0\in \TT$ there exists $ x_1\in \TT $, $|x_1-x_0|\le\exp\left(-l^\frac{\nu}{2}\right) \ll
  \rho_n $  such that
  $\log |f_n^t(x_1)|\ge nL_n-n^{1-\frac{\nu}{4}}$. On the other hand, for the subharmonic function $\log\|M_n^t\|$, its ldt (\ref{ind-ldt-fnt}) can be extended to on the strip $\TT_n:=\{z:|\Im z|\le \rho_n\}$. So,  Corollary \ref{cor:logupper} can be applied and (\ref{sup-bound-mnt}) holds. Due to (\ref{eq:fN-MN}),  we have
\begin{equation}
  \sup_{x\in \TT,|y|\le n^{-4s}}\log |f^t_n(x+iy)|\le  n L_n+n^{1-\frac{\nu}{10}}.
\end{equation}
Applying Cartan’s estimate for $f_n^t$ and  using a covering argument, we obtain that for any $H\gg 1$,
\begin{equation}
  \mes\left\{ x\in\mathbb{T}:\,\left|\log\left\Vert f_{n}^{t}\left(x\right)\right\Vert-nL_n \right|>CHn^{1-\frac{\nu}{10}}\right\} \le C\exp\left(-H\right).
\end{equation}
Combining it with (\ref{dis-un-unt}) and (\ref{Mnaxdet}), we obtain (\ref{ldt-fn-ni1}) and finish this proof.
\end{proof}

\subsection{The proof of Proposition \ref{ldt-thm}}
\begin{proof}[The proof of Proposition \ref{ldt-thm}] (\ref{ln0}), (\ref{ldt-mn-0}) and (\ref{ldt-fn-0})  make the assumptions of Lemma \ref{ind-step} hold for $n_0=(\log\lambda)^C$ with $\lambda>\lambda_0$ and small constant $\nu$. Hence, (\ref{low-up-bound-le-ni1}),  (\ref{ldt-n-i1}) and (\ref{ldt-fn-ni1}) hold for any $n\in \left[ n_0^{2B}, n_0^{5B^2}\right]$.  From now on, we choose $n_{i+1}=\left[ n_i^{2B}\right]+1$ for any $i\ge 1$. Then, due to Lemma \ref{ind-step} and the induction, we obtain that (\ref{low-up-bound-le-ni1}),  (\ref{ldt-n-i1}) and (\ref{ldt-fn-ni1}) hold for any $n\in \bigcup\limits_{i=0}^{\infty}\left[n_i^{2B}, n_i^{5B^2}\right]$. It is obvious that $n_i^{5B^2}>n_{i+1}^{2B}$. Thus, (\ref{low-up-bound-le-ni1}), (\ref{low-up-bound-le-ni1}),  (\ref{ldt-n-i1}) and (\ref{ldt-fn-ni1}) hold for any $n>n_1$.  Redefining $\nu$ and $n_0$, we obtain (\ref{posi-le-g}), (\ref{ldt-ung}) and (\ref{ldt-fng}).

At last, we apply the semialgebraic set theory to obtain the number of the intervals. Due to the semialgebraic sets theory(Subsection \ref{sec:semi-set}), there exist $v's$
truncation polynomial $\tilde v_n$ of degree less than $Cn^{4s}$, satisfying
\[ \sup_{x\in\TT}\left|\tilde v_n-v\right|\le \exp\left(-n^2\right).\]
Define the set
\[\tilde \cQ^t(x)=\left\{x:\left|\frac{1}{n}\log|\tilde f^t_{n}(x,E)|-L_{n}(E)\right|>\frac{3}{2}\delta
\right\},\]
where $\tilde f^t_{n}(x,E)=\det\left(\tilde H^t_n(x)-E\right)$ and $\tilde H^t_n(x)$ is the corresponding operators with the potential $\lambda \tilde v_n$. Obviously, $\tilde f^t_{n}(x,E)$ is a  polynomial of degree less than $Cn^{5s}$. Let
\[\cQ_1(x)=\left\{x:\left|\frac{1}{n}\log| f_{n}(x,E)|-L_{n}(E)\right|>n\delta
\right\},\]
and
\[\cQ_2(x)=\left\{x:\left|\frac{1}{n}\log| f_{n}(x,E)|-L_{n}(E)\right|>2n\delta
\right\}.\]
 it yields that
\[ \cQ_2\subset \cQ^t \subset \cQ_1.\]
Therefore, we finish this proof by applying Corollary \ref{number} with $d=1$.
\end{proof}
\begin{remark}\label{sup-bound-fnt}It is obvious that by the proofs of Proposition \ref{ldt-thm} and Corollary \ref{cor:logupper}, we obtain a better upper bound for $n>n_0$:
\begin{equation}\label{sup-bound-mn}
\sup_{x\in \TT}\log |f_n(x,E)|, \  \sup_{x\in \TT}\log\|M_n(x,E)\|,\  \sup_{x\in \TT,|y|\le n^{-4s}}\log |f^t_n(x+iy,E)|\leq n L_n(E)+Cn^{1-\nu}.\end{equation}
Actually, it also holds for $1\le n\le n_0$. Recall that for any $\lambda>\lambda_0$, $1\le n\le n_0$,
\[ |L_n(E)-\log\lambda|\lesssim (\log\lambda)^{\frac{1}{2}}, \ \ \ \ \sup_{x\in \TT}\frac{1}{n}\log\|M_n(x,E)\|\le \log\lambda+(\log\lambda)^{\frac{1}{2}}.\]
Hence, by choosing $C\sim \log\lambda$ and recalling $n_0=(\log \lambda)^{\tilde C}$, we obtain that  (\ref{sup-bound-mn}) holds for any $n\ge 1$.
\end{remark}

\section{Green Function, Wegner's estimate and spectrum criterion}
In this section, we are mainly to introduce  some common tools in the spectral theory, such as Green function, Wenger's estimate and spectrum criterion,  and apply the obtained LDTs to them to get the desired lemmas.
From now on, unless specified otherwise, we always assume $\omega\in\cD(c,A)$, $\lambda>\lambda_0$ and $n>n_0$.

\subsection{Green Function}
In this subsection, we present the key tools  to link eigenfunctions of the finite volume operators to (generalized) eigenfunctions of a large volume or in infinite volume. They are  the {\em Poisson formula} in terms of Green's function and  a bound on the off-diagonal terms of Green's function in terms of the deviations estimate for the determinant $f_n(x,E)$. It says that for any solution of the difference equation $ H(x)\phi=E\phi $, we have
\beeq \label{Poisson}
	\phi(m) = (H_{[a,b]}(x)-E)^{-1}(m, a)\phi(a-1) + (H_{[a,b]}(x)-E)^{-1}(m,b)\phi(b+1),\quad m \in [a, b].
\eneq
Easy computations show that (\ref{Poisson}) also holds if $\phi$ satisfies $H_{[c,d]}(x)\phi=E\phi $ and $[a,b]\subseteq [c,d]$. We denote this Green's function by $\cG_{[a,b]}(x,E):=\left(H_{[a,b]}(x)-E\right)^{-1}$ or $\cG_{n}(x,E):=\left(H_{n}(x)-E\right)^{-1}$. Due to Cramer's rule,
\begin{equation}
  \label{green-function}
 \cG_{n}(x,E)(k,m)=\frac{f_{k-1}(x,E)f_{n-(m+1)}(z+(m+1)\omega,E)}{f_n(x,E)}.
\end{equation}
This method was introduced into the theory of localized eigenfunctions in the fundamental work on the Anderson model by Fr\"ohlich and Spencer. Now it will help us address the relationship between the distance of an energy to the spectrum and the deviation of $f_n(x,E)$ to the Lyapunov exponent.
\begin{lemma}\label{lem:Green}  If $
    \log \big | f_n(x,E) \big | > nL_n(E) - J$, then
  \[
    \big | \cG_{[1, n]} (x,E)(j,k)\big |  \le  \exp\left(-\frac{\log\lambda}{2}|k - j|+J+Cn^{1-\nu}\right), \]
  \[\dist\left(E,\spec H_n(x)\right)\ge \exp\left(-J-Cn^{1-\nu}\right).\]
\end{lemma}
\begin{proof}
Due to (\ref{green-function}), (\ref{sup-bound-mn}) and the setting $\lambda>\lambda_0$ which makes $L_j(E)\le \log\lambda+(\log\lambda)^{\frac{1}{2}}$ for any $j$, it yields that
    \begin{eqnarray*}
    \left| \cG_{n}(x,E)(k,m)\right|&=&\frac{|f_{k-1}(x,E)|\cdot|f_{n-(m+1)}(x+(m+1)\omega,E)|}{|f_n(x,E)|}\\
    &\le &\exp\left(-\frac{\log\lambda}{2}|k-m|+J+Cn^{1-\nu}\right).
    \end{eqnarray*}
  Thus,
\[\|\cG_{n}(x,E)\|\le \exp\left(J+Cn^{1-\nu}\right),\]
which implies that
\[\dist\left(E,\spec H_n(x)\right)=\left\|\cG_{n}(x,E)\right\|^{-1}\ge \exp\left(-J-Cn^{1-\nu}\right).\]
\end{proof}

The above lemma shows that if $f_n(x,E)$ is closed to $L_n(E)$, then the Green function $ \cG_{n}(x,E)$ decays well, and the distance between $E$ and  the spectrum $\spec H_n(x)$ has an lower bound. Naturally, we also want to know what will happen when $f_n(x,E)$ is far away from $L_n(E)$.
\begin{lemma}\label{fna-small}
 If $\log |f_n(x,E)|\le nL_n(E)-J n^{1-\frac{\nu}{2}},$ then
there exists a constant $C(v) $ such that
\[\dist\left(E,\spec H_n(x)\right) <C\exp\left(-\left(J+\frac{1}{2}n^{\frac{\nu}{4}}\right)\right).\]
\end{lemma}
\begin{proof}Due to (\ref{ldt-fng}),
 \[\mes\left\{ x\in\mathbb{T}:\,\left|\log\left|f_{n}\left(x,E\right)\right|-L_n(E) \right|>n^{1-\frac{3\nu}{4}}\right\} \le \exp\left(-n^{\frac{\nu}{4}}\right).\]
 Thus, there exists $x'$ satisfying $|x'-x|<\exp\left(-n^{\frac{\nu}{4}}\right)$ such that
\begin{equation}\label{good-fng}
  \log |f_n(x',E)|>nL_n(E)-n^{1-\frac{3\nu}{4}}.
\end{equation}
Combining it with (\ref{dis-dif-eme}) and (\ref{dis-g-gn}), we have that
\begin{equation}\label{good-fng}
  \log |f^t_n(x',E)|>nL_n(E)-2n^{1-\frac{3\nu}{4}}.
\end{equation}
  Define
  \[\Psi(z)=f^t_n\left(x+\frac{10z}{\exp\left(n^{\frac{\nu}{4}} \right)|x'-x|}(x'-x),E\right),\]
  which is a complex analytic function on $\cD(0,1)$ by noting that $\exp\left(-n^{\frac{\nu}{4}}\right)\ll n^{-4s}$.
Let $z'$ be such that $\Psi(z')=f_n^t(x',E)$. Obviously, $|z'|\le \frac{1}{10}$. Due to  (\ref{sup-bound-mn}), we have that
\[\sup\limits_{x\in \TT,|y|<\exp\left(-n^{\frac{\nu}{4}}\right)}\log\|f_n^t(x+iy,E)\|\leq nL_n(E)+Cn^{1-\nu},\]
which means that $\sup\limits_{z\in \cD(0,1)}\log|\Psi(z)|<nL_n(E_0)+Cn^{1-\nu}.$
Due to the Cartan's estimate (\ref{eq:cart_bd}), we have that  there exists a set $\cB \subset \cD(0,1)$,  $\cB \in
\car_1\left(H, J\right)$, $J=CHn^{1-\frac{3\nu}{4}}$, such that
\[
  \log | \Psi(z)| >nL_n(E)-C Hn^{1-\frac{3\nu}{4}}
\]
for any $z \in \frac{1}{6} \cD(0,1)\setminus \cB$. It follows that $0\in D(z_j,r_j)\subset \cB$ with $r_j<\exp(-H)$ for some $j$, and there exists $z'\in D(z_j,r_j)$ such that $\Psi(z')=0$. Let $z''=x+\frac{10z'}{\exp\left(n^{\frac{\nu}{4}} \right)|x'-x|}(x'-x)$. Then, $E\in \spec H^t_n(z'')$ and $|z''-x|\le \exp\left(-\left(J+n^{\frac{\nu}{4}}\right)\right)$. Since $H^t_n$ is Hermitian,
\[\left\|H_n^t(z)-H_n^t(x)\right\|\le C|z-x|,\]
and that if
\[\left\|\left(H_n^t(x)-E\right)^{-1}\right\|\left\|H_n^t(z)-H_n^t(x)\right\|<1,\]
then $H_n^t(z)-E$ would be invertible. Hence, we have
\[\dist\left(E,\spec H^t_n(x)\right) <C\exp\left(-\left(J+n^{\frac{\nu}{4}}\right)\right).\]
At last, due to the fact that
\[ \sup_{x\in \TT} \left\|H^t_n(x)- H_n(x)\right\|\le \exp(-n),\]
we finish this proof.
\end{proof}

Its inverse negative proposition will play an important role in our later proof:
\begin{corollary}
  \label{spec-to-decay} If $
  \dist\left(E,\spec(H_n(x)\right)> C\exp\left(-\left(J+\frac{1}{2}n^{\frac{\nu}{4}}\right)\right)$,
	 then
\[\log\left|f_n(x,E)\right|> nL_n(E)-Jn^{1-\frac{\nu}{2}}.\]
\end{corollary}

\subsection{Wegner's estimate}
The following elementary observation links the spectra in finite volume to the decay of the Green function.

\begin{lemma}[Lemma 2.6 in \cite{GDSV}]\label{Poissonlem}
	If for any $ m\in[a,b] $, there exists
	$ \Lambda_m=[a_m,b_m]\subsetneqq [a,b] $ containing $m$ such that
	\begin{equation}\label{ob-spec}
		(1- \langle \delta_a,\delta_{a_m}\rangle) \left| \cG_{\Lambda_m}(x,E)(a_m,m) \right|
		+(1- \langle \delta_b,\delta_{b_m}\rangle) \left| \cG_{\Lambda_m}(x,E)(b_m,m) \right|
		<1,
	\end{equation}
	then $ E\notin \spec H_{[a,b]}(x) $.
\end{lemma}

We refer to the next result as the {\em covering form}  of LDT.
\begin{lemma}\label{lem:Greencoverap0}
  Suppose for each point $ m\in [1,n] $ there exists an interval $ I_m\subset [1,n] $
  such that:
  \begin{enumerate}
  \item $ \dist(m,[1,n]\setminus I_m)\ge |I_m|/100$,
  \item $ |I_m|\ge n_0$,
  \item $ \log|f_{I_m}(x,E_0)|> |I_m| L_{|I_m|}(E_0)-J_m$, where $ J_m\le|I_m|/1000$.
  \end{enumerate}
  Then
  \begin{equation*}
    \dist\left(E_0,\spec H_n(x)\right)\ge \exp\left(-\max_m\left\{J_m+ C |I_m|^{1-\nu}\right\}\right).
  \end{equation*}
\end{lemma}
\begin{proof}
Set
$$\cE(E_0):=\left\{E:|E-E_0|<\exp\left(-\max_m\left\{J_m+ C  |I_m|^{1-\nu}\right\}\right)\right\}.$$
We will apply Lemma \ref{Poissonlem} to obtain that $E\not\in \spec H_n(x)$ for any $E\in \cE(E_0)$. By the continuity of $L_n(E)$,
\[\left|L_{|I_m|}(E)-L_{|I_m|}(E_0)\right|\ll 1, \forall E\in\cE(E_0).\]
Combining it with (\ref{dis-dif-eme}), we have that
\begin{eqnarray*}
    \log|f_{|I_m|}(x,E)|&\ge& \log|f_{|I_m|}(x,E_0)|- |E-E_0|\frac{\exp \left(|I_m|L_{|I_m|}(E_0)+C|I_m|^{1-\nu}\right)}{|f_{|I_m|}(x,E_0)|}\\
    &\ge &nL_{|I_m|}(E_0)-J_m-1\ge nL_{|I_m|}(E)-J_m-2.
  \end{eqnarray*}
Thus, by  Lemma \ref{lem:Green}, it yields that
\[\left| \cG_{I_m}(x,E)(k,j)\right|\le \exp\left(-\frac{\log\lambda}{2}|k - j|+J_m+C|I_m|^{1-\nu}\right).\]
This and assumptions (1) and (2) guarantee that the assumptions of Lemma \ref{Poissonlem} are satisfied, and then we finish this proof.
\end{proof}

Now, combining it with the LDTs, Proposition \ref{ldt-thm}, we obtain the desired
Wegner's estimate.
\begin{prop}\label{Wengerforqs}
For any $(\log n)^{\frac{4}{\nu}}\le k\le n$,
\begin{equation}\label{eq:wengerforqs}
 \mes \left\{x \in \tor\::\: \dist\bigl(\spec\left(H_n(x)\right), E\bigr) < \exp\left(-k^{1-\frac{\nu}{2}}\right)\right\}
    	\le \exp\left(-k^{\frac{\nu}{4}}\right).
\end{equation}
Moreover, the set on the left-hand side is contained in the union of less than  $k^{Cs}n$ intervals.
\end{prop}
\begin{proof}
Without loss of generality, $k$ is an even number. Then,
  choose $I_m$ in Lemma \ref{lem:Greencoverap0} as
  \[I_m=\left\{\begin{array}
    {ccc}
    \ [m,k+m],&\ m\in[1,k/2],\\
    \ [m-k/2,m+k/2],&\ m\in [k/2+1, n-k/2-1],\\
    \ [m-k,m],&\ m\in[n-k/2,n]\\
  \end{array}\right.\]
and $J_m=k^{1-\frac{\nu}{2}}$. Due to (\ref{ldt-fng}), we have there exists a set $\cB_{n,E_0}$ satisfying that
\[ \mes(\cB_{n,E})< n\exp\left(-ck^{\frac{\nu}{2}}\right)\]
such that for any $x\not\in \cB_{n,E}$, all assumptions of Lemma \ref{lem:Greencoverap0} are satisfied.
The number of the intervals comes from Proposition \ref{ldt-thm} directly.
\end{proof}

An important consequence of Wegner’s estimate is that the graphs of the eigenvalues
cannot be too flat.
\begin{prop}\label{size-of-spectral-segment}
If $\cS\in \TT$ is connected and $\mes(\cS)\ge \exp\left(-k^{\frac{\nu}{4}}\right)$ for some $k$ satisfying that $\left(\log n\right)^{\frac{4}{\nu}}\le k\le n$, then
\[\mes\left(E_n^j(\cS)\right)\ge \exp\left(-k^{1-\frac{\nu}{4}}\right)\]
 for any $j\in\{1,\cdots,n\}$.
\end{prop}
\begin{proof} By   the continuity of the functions $E_j^{(n)} (x)$, $E_j^{[-n,n]}(\cS)$ is an interval.  Let $E_0$ be the center of this interval. Then if $\mes\left(E_n^j(\cS)\right)<\exp\left(-k^{1-\frac{\nu}{4}}\right)$, which means that for any $x\in \cS$,
 \[ \left|E_j^{n}(x)-E_0\right|<\exp\left(-k^{1-\frac{\nu}{4}}\right),\]
 then it contradict with (\ref{eq:wengerforqs}).
\end{proof}

\subsection{spectrum criterion}
 Lemma \ref{Poissonlem} can be also used to establish our criterion for an energy to be in the spectrum. For this we will
 use the following well-known fact.

\begin{lemma}[Lemma 2.7 in \cite{GDSV}]\label{akbk}
	If  there exist $ \delta>0 $ and sequences $ a_k\to-\infty $,
	$ b_k\to \infty $ such that
	\begin{equation*}
		\dist\left(E,\spec H_{[a_k,b_k]}(x)\right)\ge \delta,
	\end{equation*}
	then
	\begin{equation*}
		\dist\left(E,\spec H(x)\right)\ge \delta.
	\end{equation*}
\end{lemma}

Now we can  formulate the following called \textbf{spectrum criterion} lemma:
\begin{lemma}\label{spectrum-criterion}
    If  for any $ x\in\TT $, there exists $ r(x) \in [-n/2,n/2]$ such that
    \begin{equation*}
    \dist\left(E_0,\spec H_{r(x)+[-n,n]}(x)\right)\ge \exp\left(-n^{\frac{\nu}{4}}\right),
    \end{equation*}
    then
    \begin{equation*}
    	\dist(E_0,\cS_\omega)\ge \frac{1}{2}\exp\left(-n^{\frac{\nu}{4}}\right).
    \end{equation*}
\end{lemma}
\begin{proof}
	Fix $ x\in \TT $ and let $ \bar n\ge n $ be arbitrary. Let
	\begin{equation*}
		p=-\bar n-n+r(x-\bar n\omega), \ q=\bar n+n+r(x+\bar n\omega).
	\end{equation*}
	We will use Lemma \ref{Poissonlem} to show that $ \tilde E\notin \spec H_{[p,q]}(x) $ for any $ |\tilde E-E|\le \frac{1}{2}\exp\left(-n^{\frac{\nu}{4}}\right)$.
By the hypothesis and noting that
\[H_{r(x-\bar n\omega)+[-n,n]}(x-\bar n\omega)=H_{-\bar n+r(x-\bar n\omega)+[-n,n]}(x)=H_{[p,p+2n]}(x),\]
we have
	\begin{equation*}
		\dist\left(E,\spec H_{[p,p+2n]}(x)\right)\ge \exp\left(-n^{\frac{\nu}{4}}\right).
	\end{equation*}
Then,
\[\dist\left(\tilde E,\spec H_{[p,p+2n]}(x)\right)\ge \frac{1}{2}\exp\left(-n^{\frac{\nu}{4}}\right).\]
	Due to Corollary \ref{spec-to-decay} and Lemma \ref{lem:Green}, it yields that
\begin{equation*}
		\left| \cG_{[p,p+2n]}(x,\tilde E)(j,k) \right|
		\le \exp\left(-\frac{\log\lambda}{2}|j-k|+Cn^{1-\nu}+n^{1-\frac{\nu}{4}}\right).
	\end{equation*}
 It implies
\begin{equation*}
		\left| \cG_{[p,p+2n]}(x,\tilde E)(p+2n,m) \right|<1
	\end{equation*}
	for any $ m\in\left[p,p+n+[n/2]\right] $. Analogously, we can also have
	\begin{equation*}
		\left| \cG_{[q-2n,q]}(x,\tilde E)(q-2n,m) \right|<1
	\end{equation*}	
	for any $ m\in\left[q-n-[n/2],q\right] $. For $ m\in\left[p+n+[n/2],q-n-[n/2]\right] $, let
	\begin{equation*}
		a_m=m-n+r(x+m\omega)\ge p+\left[\frac{n}{2}\right]+r(x+m\omega)\ge p,\ b_m=m+n+r(x+m\omega)\le q-\left[\frac{n}{2}\right]+r(x+m\omega)\le q.
	\end{equation*}
 By the fact that $|m-a_m|,|m-b_m|\ge \frac{n}{2}$ and  $b_m-a_m=2n$, we get
	\begin{equation*}
		\left| \cG_{[a_m,b_m]}(x,\tilde E)(a_m,m) \right|
		+\left| \cG_{[a_m,b_m]}(x,\tilde E)(b_m,m) \right|<1.
	\end{equation*}
	Lemma \ref{Poissonlem} are applied to get that $ \tilde E\notin \spec H_{[p,q]}(x,\omega) $.
	Since $ \bar n $ was arbitrary, it follows that we can choose sequences $ a_k\to-\infty $ and
	$ b_k\to \infty $ such that
	\begin{equation*}
		\dist\left(E,\spec H_{[a_k,b_k]}(x,\omega)\right)\ge \frac{1}{2}\exp\left(-n^{\frac{\nu}{4}}\right).
	\end{equation*}
	Then, the conclusion follows from Lemma \ref{akbk}.
\end{proof}

\section{The Stability and Homogeneity of the Spectrum}
\subsection{The Stability of of the Spectrum}
In this subsection, we will mainly consider how much of the finite scale eigenvalue $E_j^{[-n,n]}$ survives when we pass  $n$ to the infinite.
\begin{lemma}\label{initial-spectral-segment}Let $E\in \cS_\omega$.
 There  exist $j_0\in [-n,n]$ and a segment $I$, $|I|\ge   c\exp\left(-\left(\log n\right)^{\frac{8A}{\nu}}\right)$, centered at a point $x_0$, such that
  \begin{equation}\label{EN-E0}
		\left|E_{j}^{[-n,n]}(x_0)-E\right|\le \exp\left(-n^{\frac{\nu}{4}}\right),
	\end{equation}
	and for any $ x\in I $ there exists $ \xi $, $ \|\xi\|=1 $, with support in $ [-n+1,n-1] $, such that
	\begin{equation*}
		\left\|\left(H(x)-E_{j}^{[-n,n]}(x)\right)\xi\right\|< \exp\left(-cn^{\nu}\right).
	\end{equation*}
What's more,
\[ \mes E_{j}^{[-n,n]}(I)> \exp\left(-\left(\log n\right)^{\frac{32A}{\nu^2}}\right).\]
\end{lemma}
\begin{proof} Due to the assumption that $E\in \cS_{\omega}$ and  Lemma \ref{spectrum-criterion}, there exists $ x'\in\TT $ such
	that
	\begin{equation}\label{E-spec-N-wd}
		\max_{|i|\le n/2} \dist\left(E,\spec H_{i+[-n,n]}(x')\right)< \exp\left(-n^{\frac{n}{4}}\right).
	\end{equation}
On the other hand, we apply
 Proposition \ref{Wengerforqs}  to obtain that
\underline{}	\begin{equation*}
		\left\{ x\in \TT: \dist\left(E,\spec H_{l}(x)\right) <\exp\left(-k^{1-\frac{\nu}{2}}\right )\right\}
		\subset \bigcup_{m=1}^{m_0}I_m,
	\end{equation*}
	where $l=n^{\nu}$, $k=\left(\log n\right)^{\frac{8A}{\nu}}$, $m_0\le k^{Cs}l\le l^2\ll n$ and $ I_m $ are intervals such that $ |I_m|\le \exp\left(- k^{\frac{\nu}{4}}\right)=\exp\left(- \left(\log n\right)^{2A}\right) $.
Due to the definition of  $\cD_{c,\alpha}$,
\[
 \left \|l^2\omega\right\|>\frac{c}{l^{2A}}\gg\exp\left(- \left(\log n\right)^{2A}\right) .
\]
Thus, each $ I_k $ contains at most one point of the form $ x'+(m-n)\omega$ with $0\le m<l^2$, and the same conclusion holds for $x'+(m+n-l+1)\omega$. So, we have that  there exists $|m'|\le l^2 \ll\frac{n}{2}$ such that $ x'+(m'-n)\omega $ and $  x'+(m'+n-l+1)\omega  $
	are not in any of the $ I_m $. It yields that
	\begin{equation}\label{dist-E-H-N'-wd}
  \dist(E,\spec H_{l}(x'+(m'-n)\omega)),
		\dist\left(E,\spec H_{l}(x'+(m'+n-l+1)\omega)\right)
		\ge \exp\left(-k^{1-\frac{\nu}{2}}\right )\ge \exp\left(-\left(\log n\right)^{\frac{8A}{\nu}}\right).
\end{equation}
Choosing $i=m'$ in (\ref{E-spec-N-wd}), we have that
\[ \dist\left(E,\spec H_{m'+[-n,n]}(x')\right)\leq \exp\left(-n^{\frac{n}{4}}\right).\]
which implies that  there exists $ j_0 $ such that
	\[
		\left| E-E_{j_0}^{[-n,n]}(x_0) \right|\le \exp\left(-n^{\frac{n}{4}}\right),
	\]
where $ x_0=x'+m'\omega $. Combining it with (\ref{dist-E-H-N'-wd}), it follows that
	\begin{equation}\label{dispec-wd}
		\dist\left(E_{j_0}^{[-n,n]}(x_0),\spec H_{l}(x_0-n\omega)\right),\
		\dist\left(E_{j_0}^{[-n,n]}(x_0),\spec H_{l}(x_0+(n-l+1)\omega)\right)
		\ge \frac{1}{2}\exp\left(-\left(\log n\right)^{\frac{8A}{\nu}}\right).
	\end{equation}By the fact that
\begin{equation}\label{e-eforx}
 \left |E_j^{(n)}(x_1)-E_j^{(n)}(x_2)\right|\le \left\|H_n(x_1)-H_n(x_2)\right\|\le  C|x_1-x_2|,
\end{equation}
 (\ref{dispec-wd}) also holds
	for any $ |x-x_0|\le c\exp\left(-\left(\log n\right)^{\frac{8A}{\nu}}\right)$. By the fact that $\left(\log n\right)^{\frac{8A}{\nu}}\ll l^{\frac{\nu}{4}}$,  Corollary \ref{spec-to-decay} implies that
\[\log\left|f_l(x,E)\right|> lL_l(E)-l^{1-\frac{\nu}{4}}.\]
 By Lemma \ref{lem:Green}, it obtains
 \[\left| \cG_l\left(x-n\omega, E_{j_0}^{[-n,n]}(x)\right)(i,j) \right|
		\le \exp\left(-\frac{\log\lambda}{2}|i-j|+2Cl^{1-\frac{\nu}{4}}\right).\]
Similarly,
\[\left| \cG_l\left(x+(n-l+1)\omega, E_{j_0}^{[-n,n]}(x)\right)(i,j) \right|
		\le \exp\left(-\frac{\log\lambda}{2}|i-j|+2Cl^{1-\frac{\nu}{4}}\right).\]
 Due to the Poisson formula (\ref{Poisson}), we obtain that
 \begin{equation}\label{eig-decay-l}
  \left| \phi_{j_0}^{[-n,n]}(x;m)\right|\le \exp(-cl)=\exp\left(-cn^{\nu}\right),\  |m|\ge n-\frac{l}{2}.
 \end{equation}
 Let $\bar I$ denote  the set of $x\in \mathbb{T}$ satisfying (\ref{eig-decay-l}). Obviously,
 $$I:=\left\{x:|x-x_0|\le  c\exp\left(-\left(\log n\right)^{\frac{8A}{\nu}}\right)\right\}\subset \bar I.$$
Let $\xi$ be  the normalized projection of $ \phi_{j_0}^{[-n,n]}(x) $ onto the subspace corresponding to the
	interval $ [-n+1,n-1] $. Due to the fact that $|\phi_{j}^{[-n,n]}(x;\pm n)|<\exp\left(-cn^\nu\right)$, we have that $$|\xi_m|\le 2 \exp\left(-cn^\nu\right), \ |m|>n-\frac{l}{2}.$$
 So,
	\[
		\left\|\left(H(x)-E_{j}^{[-n,n]}(x)\right)\xi\right\|=\|\xi_{-n+1}\|+\|\xi_{n-1}\|<4 \exp\left(-cn^\nu\right).
	\]
At last, applying  Proposition \ref{size-of-spectral-segment}, we obtain that
\[ \mes E_{j}^{[-n,n]}(I)\ge \exp\left(-\left(\log n\right)^{\frac{32A}{\nu^2}}\right).\]
\end{proof}

Next we address the stability of the spectral segments produced via the previous lemma by induction on scales. The inductive step that will be stated in Lemma \ref{stabilization-approx-ef} is essentially similar to  Lemma~12.22 of \cite{B05} and Lemma 3.3 of \cite{GDSV}. Compared to them, we will use the following elementary lemma from basic perturbation theory to shorten ours:
\begin{lemma}[Lemma 2.40 in \cite{GSV16}]\label{lem:eigenvector-stability} Let
  $A $ be a $ N\times N $ Hermitian matrix.  Let
  $E,\epsilon \in \mathbb{R}$, $ \epsilon>0 $, and suppose there exists
  $\phi\in \mathbb{R}^N$, $\|\phi\|=1$, such that
  \begin{equation}\label{eq:2contraction1}
    \begin{split}
      \left\|(A-E)\phi\right\|< \varepsilon.
    \end{split}
  \end{equation}Then,
  there exists a normalized eigenvector $\psi$ of $A$ with an eigenvalue $E_0$
   such that
  \begin{equation}\label{eq:2contraction1aaaM}
    \begin{split}
    E_0\in (E-\varepsilon\sqrt2,E+\varepsilon\sqrt 2),\\
      \left|\langle \phi,\psi\rangle\right|\ge (2N)^{-1/2}.
    \end{split}
  \end{equation}
\end{lemma}

\begin{lemma}\label{stabilization-approx-ef}
	 Assume that $ E_j^{[-n,n]}(I,\omega)\subset (E',E'') $ and  for each $ x\in I $, there exists $ \xi $,
	$ \|\xi\|=1 $, with support in $ [-n+1,n-1] $, such that
	\begin{equation}\label{approx-ef}
		\left\|\left(H(x)-E_j^{[-n,n]}(x)\right)\xi\right\|<\exp\left(-cn^{\nu}\right).
	\end{equation}
	 Then, there exists $n_1$ satisfying $n_1\sim n^{100}$ such that
	 we can partition $ I $ into intervals $ I_m $,
	$ m\le n_1^{C_0s} $, where $C_0=C_0(v)$, and for each $ I_m $, there exists $ j_1\in[-n_1,n_1] $ such that
	\begin{equation*}
		\left| E_j^{[-n,n]}(x)-E_{j_1}^{[-n_1,n_1]}(x) \right|\le 2\exp\left(-cn^{\nu}\right), x\in I_m,
	\end{equation*}
	and for each $ x\in I_m $, there exists $ \xi_1 $, $ \|\xi_1\|=1 $, with support in $ [-n_1+1,n_1-1] $, satisfying
	\begin{equation*}
		\left\|\left(H(x)-E_{j_1}^{[-n_1,n_1]}(x)\right)\xi_1\right\|<\exp\left(-\frac{c}{2}n_1^{\nu}\right).
	\end{equation*}	
\end{lemma}
\begin{proof}
	Fix $ x\in I $ and let $ \xi $ be as in (\ref{approx-ef}). Due to the fact that
\[H_{[-n,n]}(x)\xi = H_{[-n_1,n_1]}(x)\xi,\]
we have
\[\left\|\left(H_{[-n_1,n_1]}(x)-E_j^{[-n,n]}(x)\right)\xi\right\|<\exp\left(-cn^{\nu}\right).\]
By Lemma \ref{lem:eigenvector-stability}, it yields that there exists $ j_1(x)\in[-n_1,n_1] $ such that
\begin{equation}\label{Ej-Ej1}
		\left|E_j^{[-n,n]}(x)-E_{j_1}^{[-n_1,n_1]}(x)\right|<\sqrt{2}\exp\left(-cn^{\nu}\right),
	\end{equation}
\begin{equation}\label{psi-mass}
		\left| \langle \xi, \phi_{j_1}^{[-n_1,n_1]}(x)\rangle \right|\ge 1/\sqrt{2n_1}.
	\end{equation}
It is the same to the proof of Lemma \ref{initial-spectral-segment} that, there exists $0< k'_1<l^2_1<\frac{n_1}{2}$ such that $ x'+(k'_1-n_1)\omega $ and $  x'+(-k_1+n_1-l_1+1)\omega  $
	are not in any of the $ I_k $, where $ |I_k|\le \exp\left(- (\log n_1)^{2A}\right) $ and  comes from
	\begin{equation*}
		\left\{ x\in \TT: \dist\left(E,\spec H_{l_1}(x)\right) <\exp\left(-k_1^{1-\frac{\nu}{2}}\right )\right\}
		\subset \bigcup_{k=1}^{m_1}I_k.
	\end{equation*}	
Similarly,
	\begin{align*}
		& \left| \cG_{[k'_1-n_1,k'_1-n_1+l_1]}(x,E_{j_1}^{[-n_1,n_1]})(i,j) \right|
		\le \exp\left(-\frac{\log\lambda}{2}|i-j|+2Cl_1^{1-\frac{\nu}{4}}\right).\\
		& \left| \cG_{[-k'_1+n_1-l,-k'_1+n_1]}(x,E_{j_1}^{[-n_1,n_1]})(i,j) \right|
		\le\exp\left(-\frac{\log\lambda}{2}|i-j|+2Cl_1^{1-\frac{\nu}{4}}\right).
	\end{align*}
and
\begin{equation}\label{psi-decay-wd}
		\left|\phi_{j_1}^{[-n_1,n_1]}(x;j)\right|\le \exp\left(-cn_1^{\nu}\right), \  n_1-k'_1-\frac{3l_1}{4}\le |j| \le n_1-k'_1-\frac{l_1}{4}.
	\end{equation}
Let $k''_1=n_1-k'_1-\frac{l_1}{4}$ and  $ \eta $ be the normalized projection of $ \phi_{j_1}^{[-n_1,n_1]}(x) $ onto the subspace corresponding to the
	interval $ [-k_1'',k_1''] $. Due to the fact that $k_1''\gg n$ and  (\ref{psi-mass}), it implies that
\[\sum_{m=-k}^k\left|\phi_{j_1}(x,m)\right|^2\ge \sum_{m=-n}^n\left|\phi_{j_1}(x,m)\right|^2\ge \frac{1}{n_1}.\]
Therefore, for any $k_1''-\frac{l_1}{2}\le |m| \le k_1''$,
\begin{equation}\label{eta-decay-wd}
		|\xi_1(m)|\le \sqrt{n_1}\exp\left(-cn_1^{\nu}\right)\le \exp\left(-\frac{c}{2}n_1^{\nu}\right),
	\end{equation}
	\begin{equation}\label{existence-of-approx-ef-wd}
		\left\|\left(H(x)-E_{j_1}^{[-n_1,n_1]}(x)\right)\xi_1\right\|\le \exp\left(-\frac{c}{2}n_1^{\nu}\right).
	\end{equation}

Now we can declare that we have finished this proof. The reason is that the last thing for our proof is to estimate the number of components of the set of phases $ x $ that satisfy (\ref{Ej-Ej1}) and (\ref{existence-of-approx-ef-wd}).  The main idea is to use the semialgebraic set theory to approximate  the potential $ \lambda v $  by a trigonometric polynomials with the fixed $\omega$, just as what we have done to obtain the number of the intervals in the proof of Proposition \ref{ldt-thm}. By recalling the definitions of $\tilde v_n$ and $\tilde H^t(x)$, we have that
\[\left\|H(x)-\tilde H^t(x)\right\|\le \lambda \sup_{x\in\TT}\left|\tilde v_n-v\right|\le \exp(-n).\]
Then, the method in the proofs of Lemma 4.3 in \cite{GDSV} and  Lemma 12.22 in \cite{B05} can been applied here directly.
\end{proof}

 Now we complete this induction and  obtain the stability of the spectral segments.
\begin{lemma}\label{finite-segment-to-full-spectrum-Bourgain}
	Let $ I\subset [0,1] $ be an interval and
	let $ j\in[-n,n] $. Assume that $ E_j^{[-n,n]}(I,\omega)\subset (E',E'') $ and that for each $ x\in I $, there exists $ \xi $,
	$ \|\xi\|=1 $, with support in $ [-n+1,n-1] $, such that
	\begin{equation}\label{approx-ef-bis}
			\left\|\left(H(x)-E_{j}^{[-n,n]}(x)\right)\xi\right\|<\exp\left(-cn^\nu\right).
	\end{equation}
	Then
	\begin{equation*}
		\mes\left(E_j^{[-n,n]}(I)\setminus \cS_\omega\right)< \exp\left(-\frac{c}{4}n^\nu\right).
	\end{equation*}
\end{lemma}
\begin{proof}
	Let  $ n_1=n^{100} $. Using Lemma \ref{stabilization-approx-ef}, we
	partition $ I $ into intervals $ I_m $,
	$ m\le n_1^{Cs} $, and for each $ I_m $, there exists $ j_1\in[-n_1,n_1] $ such that
	\begin{equation}\label{Ej-Ej1-bis}
		\left| E_j^{[-n,n]}(x)-E_{j_1}^{[-n_1,n_1]}(x) \right|\le 2\exp\left(-cn^\nu\right), x\in I_m,
	\end{equation}
	and for each $ x\in I_m $, there exists $ \xi $, $ \|\xi\|=1 $, with support in $ [-n_1+1,n_1-1] $, satisfying
	\begin{equation}\label{approx-ef-N1}
		\left\|\left(H(x)-E_{j_1}^{[-n_1,n_1]}(x)\right)\xi\right\|<\exp\left(-\frac{c}{2}n^\nu\right).
	\end{equation}	
 	Let
	\begin{equation*}
		\cE_{n,1}= \bigcup_m\left( E_j^{[-n,n]}(I_m)\ominus E_{j_1}^{[-n_1,n_1]}(I_m)\right),
	\end{equation*}
	where $ \ominus $ denotes the symmetric set difference
\[\mathcal{S}_1\ominus \mathcal{S}_2:=\left(\mathcal{S}_1-\mathcal{S}_2\right) \bigcup\left(\mathcal{S}_2-\mathcal{S}_1\right).\]
By the continuity of the parametrization of the eigenvalues, $m<n_1^{Cs}$, the setting $n_1=n^{100}$ and
	(\ref{Ej-Ej1-bis}), it follows that
$$ \mes(\cE_{n,1})\le \exp\left(-\frac{c}{3}n^\nu\right).$$
	
	Note that (\ref{approx-ef-bis}) implies that $ \dist(E,\cS_\omega)<\exp\left(-cn^\nu\right) $ for all $ E\in E_j^{[-n,n]}(I) $.
	At the same time, if $ E\in E_j^{[-n,n]}(I)\setminus \cE_{n,1} $,
	then $ E\in E_{j_1}^{[-n_1,n_1]}(I_m) $ for some $ m $, and (\ref{approx-ef-N1}) implies that
	$$ \dist(E,\cS_\omega)<\exp\left(-\frac{c}{2}n^\nu\right).$$

Let $ n_k=n^{100^k} $. Through iteration we obtain sets $ \cE_{n,k} $ such that
	$$ \mes(\cE_{n,k})<\exp\left(-\frac{c}{4}n^\nu\right),$$ and if
	$ E\in E_j^{[-n,n]}(I)\setminus \bigcup_{l\le k}\cE_{n,l} $, then
	$$ \dist(E,\cS_\omega)<\exp\left(-\frac{c}{3}n^\nu\right).$$ Therefore, we finish this proof by noting that
	\begin{equation*}
		E_{j_0}^{[-n,n]}(I)\setminus \cS_\omega\subset \bigcup_k \cE_{n,k}.
	\end{equation*}
\end{proof}

\subsection{The proof of the homogeneity}
\begin{proof}[The proof of Theorem \ref{main-thm}]   Lemma \ref{initial-spectral-segment} implies that there exist  $ j\in[-n,n] $ and of a segment $ I $,
	$ |I|>c\exp\left(-\left(\log n\right)^{\frac{8A}{\nu}}\right) $, centered at a point $ x_0 $, such that
	\begin{equation}\label{EN-E0}
		|E_{j}^{[-n,n]}(x_0)-E|\le\exp\left(-n^{\frac{\nu}{4}}\right),
	\end{equation}
	and for any $ x\in I $ there exists $ \xi $, $ \|\xi\|=1 $, with support in $ [-n+1,n-1] $, such that
	\begin{equation*}
		\left\|\left(H(x)-E_{j}^{[-n,n]}(x)\right)\xi\right\|< \exp\left(-cn^{\nu}\right).
	\end{equation*}
	Then, we can apply Lemma \ref{finite-segment-to-full-spectrum-Bourgain} to get
	\begin{equation*}
		\mes\left(E_j^{[-n,n]}(I)\setminus\cS_\omega\right)\le \exp\left(-\frac{c}{4}n^{\nu}\right).
	\end{equation*}
Recall that
\[ \mes E_{j}^{[-n,n]}(I)> \exp\left(-\left(\log n\right)^{\frac{32A}{\nu^2}}\right).\]
Hence, we choose $\sigma_0=\exp\left(-\left(\log n_0\right)^{\frac{32A}{\nu^2}}\right)$. Then, for any $\sigma<\sigma_0$, there exists some $n>n_0$ such that $\exp\left(-\left(\log n\right)^{\frac{32A}{\nu^2}}\right)\sim \sigma$, and
\begin{eqnarray*}
  &&\left|\left(E_0-\exp\left(-\left(\log n\right)^{\frac{32A}{\nu^2}}\right),E_0+\exp\left(-\left(\log n\right)^{\frac{32A}{\nu^2}}\right)\right)\cap\cS_\omega\right|\\
  &\ge& \exp\left(-\left(\log n\right)^{\frac{32A}{\nu^2}}\right)-\exp\left(-n^{\frac{\nu}{4}}\right)-\exp\left(-\frac{c}{4}n^{\nu}\right)\\
&\ge& \frac{1}{2}\exp\left(-\left(\log n\right)^{\frac{32A}{\nu^2}}\right).
\end{eqnarray*}
It implies that $\diam(\cS_{\omega})>0$ and  the homogeneity holds for $\tau=\left\{\begin{array}
  {cc}
  \frac{1}{4},& \sigma\le \sigma_0,\\
  \frac{\sigma_0}{4\diam(\cS_{\omega})},  &  \sigma>\sigma_0.
\end{array}\right..$
\end{proof}

\end{document}